%% file: main.tex
\documentclass[10pt,twoside,fleqn]{imsart}

 \setattribute{fline}       {cmd}  {\vskip22\p@
                                     \hrule}
  \setattribute{lline}       {cmd}  {\vskip10\p@
                                     \hrule}

 \def\ims@thmshape{3}

\makeatletter
 \set@page@layout{39pc}{622pt}
\makeatother

 \PassOptionsToPackage{fleqn}{amsmath}
  \setattribute{thebibliography}{size}{\small}

\usepackage[utf8]{inputenc}
\usepackage{amssymb}
\usepackage{makeidx}
\usepackage[english]{babel}
\usepackage{graphicx}
\usepackage{amsfonts}
\usepackage{lmodern}
\usepackage{oldgerm}
\usepackage{mathrsfs}
\usepackage[active]{srcltx}
\usepackage{amsthm}
\usepackage{verbatim}
\usepackage{enumerate}
\usepackage{amsmath}
\usepackage[numbers]{natbib}
\usepackage[toc,page]{appendix}
\usepackage{aliascnt,bbm}
\usepackage{array}
\usepackage{hyperref}
\usepackage[ruled,vlined]{algorithm2e}
\usepackage[textwidth=4cm,textsize=footnotesize]{todonotes}
\usepackage{xargs}
\usepackage{color}

\makeatletter
\newcommand\ackname{Remerciements}
\if@titlepage
  \newenvironment{acknowledgements}{%
      \titlepage
      \null\vfil
      \@beginparpenalty\@lowpenalty
      \begin{center}%
        \bfseries \ackname
        \@endparpenalty\@M
      \end{center}}%
     {\par\vfil\null\endtitlepage}
\else
  
\fi
\makeatother

\usepackage{aliascnt}
\theoremstyle{plain}
\newtheorem{theorem}{Theorem}
\newtheorem{assumption}{H\hspace{-3pt}}
\newtheorem*{assumption*}{A}

\newtheorem{assumptionAR}{AR\hspace{-3pt}}
\newtheorem{assumptionCN}{CN\hspace{-3pt}}

\newaliascnt{proposition}{theorem}
\newtheorem{proposition}[proposition]{Proposition}
\aliascntresetthe{proposition}

\newaliascnt{lemma}{theorem}
\newtheorem{lemma}[lemma]{Lemma}
\aliascntresetthe{lemma}
\newaliascnt{corollary}{theorem}
\newtheorem{corollary}[corollary]{Corollary}
\aliascntresetthe{corollary}

\theoremstyle{definition}
\newaliascnt{definition}{theorem}

\aliascntresetthe{definition}

\newaliascnt{remark}{theorem}

\aliascntresetthe{remark}

\theoremstyle{remark}

\newcommand{\1}{\mathbbm{1}}

\newcommand{\B}{\mathcal{B}}

\DeclareMathOperator*{\pCN}{cn}
\newcommand{\tildegamma}{\widetilde{\gamma}}

\newcommand{\inv}[1]{\frac{1}{#1}}

\newcommand{\N}{\ensuremath{\mathbb{N}}}

\newcommand{\R}{\ensuremath{\mathbb{R}}}
\newcommand{\rset}{\ensuremath{\mathbb{R}}}

\newcommand{\E}{\mathbb{E}}

\newcommand{\Kfrac}{\mathscr{K}}

\newcommand{\abs}[1]{\left\vert #1 \right\vert}
\newcommand{\norm}[1]{\Vert #1 \Vert}

\newcommand{\parenthese}[1]{(#1)}
\newcommand{\parentheseDeux}[1]{\left[ #1 \right]}
\newcommand{\defEns}[1]{\left\lbrace #1 \right\rbrace }

\newcommand{\eqdef}{\overset{\text{\tiny def}} =}

\newcommandx\probaMarkovTilde[2][2=]
{\ifthenelse{\equal{#2}{}}{{\widetilde{\mathbb{P}}_{#1}}}{\widetilde{\mathbb{P}}_{#1}\left[ #2\right]}}
\newcommand{\probaMarkov}[2]{\mathbb{P}_{#1}\left[ #2\right]}

\newcommand{\expeMarkov}[2]{\mathbb{E}_{#1} \left[ #2 \right]}
\newcommand{\probaMarkovTildeDeux}[2]{\widetilde{\mathbb{P}}_{#1} \left[ #2 \right]}
\newcommand{\expeMarkovTilde}[2]{\widetilde{\mathbb{E}}_{#1} \left[ #2 \right]}
\newcommand{\sachant}[1]{\left| #1 \right.}

\renewenvironment{proof}[1][{\textit{Proof:}}]{\begin{trivlist} \item[\em{\hskip \labelsep #1}]}{\ensuremath{\qed} \end{trivlist}}

\newcommand{\couplage}[2]{\mathcal{C}(#1,#2)}
\newcommand{\Pens}{\mathcal{P}}

\newcommand{\phiV}{\phi \circ V}

\newcommand{\smallSet}{C}
\newcommand{\smallSetD}{D}
\newcommand{\filtration}{\mathcal{F}}
\newcommand{\filtrationTilde}{\widetilde{\mathcal{F}}}

\newcommand{\dint}{\mathrm{d}}

\newcommand{\Vl}{{\mathcal{V}}}

\newcommand{\flecheLimite}{\underset{n\to+\infty}{\longrightarrow}}

\newcommand{\plusinfty}{+\infty}

\newcommand{\hilbert}{\mathcal{H}}

\newcommand{\covariance}{\mathsf{C}}
\newcommand{\setAbruit}{\mathscr{I}}

\newcommand{\setaccept}{\mathscr{A}}
\newcommand{\setacceptdeux}{\tilde{\mathscr{A}}}
\newcommand{\setreject}{\mathscr{R}}

\def\ie{\textit{i.e.}}
\def\eqsp{\;}
\newcommand{\coint}[1]{\left[#1\right)}
\newcommand{\ocint}[1]{\left(#1\right]}
\newcommand{\ooint}[1]{\left(#1\right)}
\newcommand{\ccint}[1]{\left[#1\right]}
\def\inv{\leftarrow}

\usepackage{hyperref}
\usepackage{cleveref}

\newcommand{\boule}[2]{\operatorname{B}(#1,#2)}
\newcommand{\boulefermee}[2]{\overline{\operatorname{B}}(#1,#2)}

\newcommand{\deta}{d_{\eta}}

\def\as{\ensuremath{\text{a.s}}}
\def\rmd{\mathrm{d}}
\newcommandx\sequence[3][2=,3=]
{\ifthenelse{\equal{#3}{}}{\ensuremath{\{ #1_{#2}\}}}{\ensuremath{\{ #1_{#2}, \eqsp #2 \in #3 \}}}}
\newcommandx{\sequencen}[2][2=n\in\N]{\ensuremath{\{ #1, \eqsp #2 \}}}
\newcommandx\sequenceDouble[4][3=,4=]
{\ifthenelse{\equal{#3}{}}{\ensuremath{\{ (#1_{#3},#2_{#3}) \}}}{\ensuremath{\{  (#1_{#3},#2_{#3}), \eqsp #3 \in #4 \}}}}
\newcommandx{\sequencenDouble}[3][3=n\in\N]{\ensuremath{\{ (#1_{n},#2_{n}), \eqsp #3 \}}}
\newcommand{\wrt}{w.r.t.}

\def\iid{i.i.d.}

\newcommand{\nset}{\mathbb{N}}
\newcommand{\zset}{\mathbb{Z}}
\def\rme{\mathrm{e}}

\begin{document}

\begin{frontmatter}

  \title{Subgeometric rates of convergence in Wasserstein distance for Markov
    chains} \runtitle{Subgeometric rates of convergence}

\begin{aug}
  \author{\fnms{Alain} \snm{Durmus}\thanksref{t1}     \ead[label=e1]{alain.durmus@telecom-paristech.fr}},
  \author{\fnms{Gersende} \snm{Fort} \thanksref{t2}   \ead[label=e2]{gersende.fort@telecom-paristech.fr}},
  \and
  \author{\fnms{\'Eric} \snm{Moulines} \thanksref{t3} \ead[label=e3]{eric.moulines@telecom-paristech.fr}}
  \address{\thanksref{t1} \thanksref{t2} \thanksref{t3} LTCI, Telecom ParisTech \& CNRS, 46 rue Barrault, 75634 Paris Cedex 13, France}

 \thankstext{t1}{\printead{e1}}
 \thankstext{t2}{\printead{e2}}
 \thankstext{t3}{\printead{e3}}
\affiliation{\thanksref{t2} CNRS LTCI ; Télécom ParisTech}
\affiliation{ \thanksref{t3} Institut Mines-Télécom ; Télécom ParisTech ; CNRS LTCI }
\runauthor{A. Durmus, G. Fort \and \'E Moulines}

\end{aug}

\maketitle

\begin{abstract}
  { In this paper, we provide sufficient conditions for the existence of the
    invariant distribution and for subgeometric rates of convergence in Wasserstein
    distance for general state-space Markov chains which are (possibly) not
    irreducible.  Compared to \cite{butkovsky:2012}, our approach is based on a
    purely probabilistic coupling construction which allows to retrieve rates
    of convergence matching those previously reported for convergence in total
    variation in \cite{Douc07computableconvergence}.

    Our results are applied to establish the subgeometric ergodicity in
    Wasserstein distance of non-linear autoregressive models and of the
    pre-conditioned Crank-Nicolson Markov chain Monte Carlo algorithm in
    Hilbert space. }
\end{abstract}

\begin{keyword}[class=MSC]
  \kwd{60J10}
  \kwd{60B10}
  \kwd{60J05}
  \kwd{60J22}
  \kwd{65C40}
\end{keyword}

\begin{keyword}
\kwd{Markov chains}
\kwd{Wasserstein distance}
\kwd{Subgeometric ergodicity}
\kwd{Markov chain Monte Carlo  in infinite dimension}
\end{keyword}
\end{frontmatter}

\def\sectionautorefname{Section}
\def\subsectionautorefname{Section}
\def\subsubsectionautorefname{Section}
\input{intro}

\input{cadre.tex}

\input{application.tex}
\input{section_proof.tex}
\input{ProofApplication.tex}
\appendix
\input{appendice.tex}

\bibliographystyle{plain}
\bibliography{biblio}

\end{document}

%% file: intro.tex
\section{Introduction}
Convergence of general state-space Markov chains in total variation distance
(or $V$-total variation) has been studied by many authors. There is a wealth of
contributions establishing explicit rates of convergence under conditions
implying geometric ergodicity; see \cite[Chapter~16]{bible},
\cite{roberts:rosenthal:2004}, \cite{baxendale:2005},
\cite{butkovsky:veretennikov:2013} and the references therein. Subgeometric (or
Riemanian) convergence has been more scarcely studied; \cite{Tuominen_subgeo}
characterized subgeometric convergence using a sequence of drift conditions,
which proved to be difficult to use in practice. \cite{jarner:roberts:2002}
have shown that, for polynomial convergence rates, this sequence of drift
conditions can be replaced by a single drift condition, which shares some
similarities with the classical Foster-Lyapunov approach for the geometric
ergodicity.  This result was later extended by \cite{fort:moulines:2003} and
\cite{douc:fort:moulines:soulier:2004} to general subgeometric rates of
convergence. Explicit convergence rates were obtained in
\cite{vertennikov:1997,fort:moulines:2003,douc:guillin:moulines:2008} and
\cite{andrieu:fort:vihola:2014}.

The classical proofs of convergence in total variation distance are based either
on a regenerative or a pairwise coupling construction, which requires the existence of
accessible small sets and additional assumptions to control the moments of the
successive return time to these sets. The existence of an accessible small set
implies that the chain is irreducible.

In this paper, we establish rates of convergence for general state-space Markov
chains which are (possibly) not irreducible.  In such cases, Markov chains
might not converge in total variation distance, but nevertheless may converge
in a weaker sense; see for example \cite{madras:sezr:2010}.  We study in this
paper the convergence in Wasserstein distance, which also implies the weak
convergence. The use of the Wasserstein distance to obtain explicit rates of
convergence has been considered by several authors, most often under conditions
implying geometric ergodicity. A significant breakthough in this domain has
been achieved in \cite{WeakHarris} 
.  The main motivation of \cite{WeakHarris} was the convergence of the
solutions of stochastic delay differential equations (SDDE) to their invariant
measure. Nevertheless, the techniques introduced in \cite{WeakHarris} laid the
foundations of several contributions. \cite{hairer:stuart:vollmer:2012} used
these techniques to prove the convergence of Markov chain Monte Carlo
algorithms in infinite dimensional Hilbert spaces.  An application for switched
and piecewise deterministic Markov processes can be found in
\cite{cloez:hairer}. The results of \cite{WeakHarris} were generalized by
\cite{butkovsky:2012} which establishes conditions implying the existence and
uniqueness of the invariant distribution, and the subgeometric ergodicity of
Markov chains (in discrete-time) and Markov processes (in continuous-time).
\cite{butkovsky:2012} used this result to establish subgeometric ergodicity of
the solutions of SDDE.  Nevertheless, when applied to the
  context of $V$-total variation, the rates obtained in \cite{butkovsky:2012}
in discrete-time do not exactly match the rates established in
\cite{douc:fort:moulines:soulier:2004}.

In this paper, we complement and sharpen the results presented in
\cite{butkovsky:2012} in the discrete-time setting. The approach developed in
this paper is based on a coupling construction, which shares some similarities
with the pairwise coupling used to prove geometric convergence in $V$-total
variation.  The arguments are therefore mostly probabilistic whereas
\cite{butkovsky:2012} heavily relies on functional analysis techniques and
methods.  We provide a sufficient condition couched in terms of a single drift
condition for a coupling kernel outside an appropriately defined coupling set,
extending the notion of $d$-small set of \cite{WeakHarris}. We then show how
this single drift condition implies a sequence of drift inequalities from which
we deduce an upper bound of some subgeometric moment of the successive return
times to the coupling set. The last step is to show that the Wasserstein
distance between the distribution of the chain and the invariant probability
measure is controlled by these moments.  We apply our results to the
convergence of some Markov chain Monte Carlo samplers with heavy tailed target
distribution and to nonlinear autoregressive models whose the noise
distribution can be singular with the Lebesgue measure. We also study the
convergence of the preconditioned Crank-Nicolson algorithm when the target
distribution has a density \wrt\ a Gaussian measure on an Hilbert space,
under conditions which are weaker than \cite{hairer:stuart:vollmer:2012}.

The paper is organized as follows: in \autoref{sec:main_results}, the main
results on the convergence of Markov chains in Wasserstein distance are
presented, under different sets of assumptions. \autoref{sec:application} is
devoted to the applications of these results.  The proofs are given in
\autoref{sec:proof} and \autoref{sec:proof:application}.

\subsection*{Notations}
Let $(E,d)$ be a Polish space where $d$ is a distance bounded by $1$.
We denote by $\B(E)$ the associated Borel $\sigma$-algebra and $\Pens(E)$ the set of
probability measures on $(E, \B(E))$.
Let $\mu,\nu \in \Pens(E)$;  $\lambda$ is a coupling of $\mu$ and $\nu$ if $\lambda$ is a probability on the product space $(E \times E, \B(E \times E))$, such that $\lambda(A \times E) = \mu(A)$ and $\lambda(E \times A)= \nu(A)$ for all $A \in \B(E)$.
The set of couplings of $\mu,\nu \in \Pens(E)$ is denoted $\couplage{\mu}{\nu}$. Let $P$ be Markov kernel of $E \times \B(E)$; a Markov kernel $Q$ on $(E \times E, \B(E\times E))$ such that, for every $x,y \in E$, $Q((x,y),\cdot)$ is a coupling of $P(x,\cdot)$ and $P(y,\cdot)$ is a \emph{coupling kernel} for $P$.

The Wasserstein metric associated with $d$, between two probability measures
$\mu, \nu \in \Pens(E)$  is defined by:
\begin{equation}
\label{def_wasser}
W_{d} (\mu,\nu) = \inf_{ \gamma \in \couplage{\mu}{\nu}} \int_{E\times E}d(x,y) \dint \gamma(x,y) \eqsp.
\end{equation}
When $d$ is the trivial metric $d_0(x,y) = \1_{x \neq y}$, the associated
Wasserstein metric is the total variation distance
$W_{d_0}(\mu,\nu) =  \sup_{A\in \B(E)} \abs{\mu(A)-\nu(A)}$.
Since $d$ is bounded, the Monge-Kantorovich duality Theorem implies (see \cite[Remark
6.5]{VillaniTransport}) that the lower bound in \eqref{def_wasser} is realized.
In addition, $W_{d}$ is a metric on $\Pens(E)$
and  $\Pens(E)$ equipped with $W_{d}$ is a Polish space; see
\cite[Theorems~6.8 and ~6.16]{VillaniTransport}.  Finally, the convergence in $W_{d}$ is
equivalent to the  weak convergence, since $W_d$ is equivalent to the Prokorov metric (see e.g.  \cite[Theorem 6.8 and 6.9]{billingsley:1999}).

Let $\Lambda_0$ be the set of measurable functions $ r_0 : \R_+ \rightarrow
\coint{2,+\infty}$, such that $r_0$ is non-decreasing, $x \mapsto \log (
r_0(x)) / x$ is non-increasing and $\lim_{x \to \infty} \log(r_0(x))/x = 0$.
Denote by $\Lambda$ the set of positive functions $r : \R_+ \rightarrow
\ooint{0,\plusinfty}$, such that there exists $r_0 \in \Lambda_0$ satisfying:
\begin{equation}
\label{eq:subgeometric-function}
0 < \liminf_{x \rightarrow + \infty } r(x) / r_0(x) \leq
\limsup_{x \rightarrow + \infty } r(x) / r_0(x) < + \infty \eqsp.
\end{equation}
Finally, let $\mathbb{F}$ be the set of concave increasing functions $\phi :
\R_+ \to \R_+$, continuously differentiable on $\coint{1,\plusinfty}$, and
satisfying $\lim_{x \to +\infty} \phi(x)= +\infty$ and $\lim_{x \to +\infty}
\phi'(x) =0$. For $\phi \in \mathbb{F}$, we denote by $\phi^\inv$ the inverse
of $\phi$.


%% file: cadre.tex
\section{Main results}
\label{sec:main_results}
The key ingredient for the derivation of the convergence of a Markov kernel $P$ on $(E,d)$ is the existence of a coupling kernel $Q((x,y),\cdot)$ for $P$
satisfying a strong contraction property when $(x,y)$ belongs to a set $\Delta$, referred to as a \emph{coupling set}.
For $\Delta \in \B(E \times E)$, a positive integer $\ell$ and $\epsilon > 0$, consider the following assumption:
\begin{assumption}[$\Delta, \ell,\epsilon$]
\label{hyp:small-set}
\begin{enumerate}[(i)]
\item \label{item:contract_coupling_2} $Q$ is a $d$-weak-contraction: for every $x,y \in E$, $Qd(x,y) \leq d(x,y)$.
\item \label{item:contract_coupling_3} $Q^\ell d(x,y) \leq (1-\epsilon)d(x,y)$,
  for every $(x,y) \in \Delta$.
\end{enumerate}
\end{assumption}
A set $\Delta$ satisfying
\autoref{hyp:small-set}($\Delta,\ell,\epsilon$)-\eqref{item:contract_coupling_3}
will be referred to as a $(\ell,\epsilon,d)$-coupling set. Of course the
definition of this set also depends on the choice of the coupling kernel $Q$,
but this dependence is implicit in the notation.  If $d=d_0$ and $\Delta$ is a 
$(1,\epsilon)$-pseudo small set (with $\epsilon > 0$) in the sense that 
\[
\inf_{(x,y) \in \Delta} [P(x,\cdot) \wedge P(y,\cdot)](E) \geq \epsilon \eqsp, 
\]
then \autoref{hyp:small-set}($\Delta,1 ,\epsilon$) is satisfied by the pairwise coupling
kernel (see \cite{roberts:rosenthal:2001}). Furthermore, a simple way to check that $\Delta \in \B(E \times E)$ is a
$(1,\epsilon,d)$-coupling set is the following. Let $\epsilon >0$. If for all $(x, y) \in E \times E$, $W_d (P (x, \cdot), P (y, \cdot)) \leq d(x, y)$, and for
all $(x, y) \in \Delta $, $W_d (P (x, \cdot), P (y, \cdot)) \leq (1 - \epsilon)d(x, y)$, then \cite[ corollary 5.22]{VillaniTransport}
implies that there exists a Markov kernel $Q$ on $(E \times E, \B(E \times E))$ satisfying
\autoref{hyp:small-set}($\Delta,1,\epsilon$).

The following theorem shows that, under
\autoref{hyp:small-set}($\Delta,\ell,\epsilon$) and a condition which
essentially claims that if the first moment of the hitting time to the coupling
set $\Delta$ is finite, the Markov kernel $P$ admits a unique invariant
distribution.

\begin{theorem}
\label{theo:existence-pi}
Assume that there exist 
\begin{enumerate}[(i)]
\item a coupling kernel $Q$ for $P$, a set $\Delta \in \B(E \times E)$, $\ell
  \in \nset^*$ and $\epsilon > 0$ such that \autoref{hyp:small-set}($\Delta,
  \ell, \epsilon$) holds,
\item  a measurable function $\Vl: E^2 \to \coint{1,\infty}$ and a constant $b < \infty$ such that the following drift condition is satisfied.
\begin{equation}
\label{eq:simple-drift-condition}
Q \Vl(x,y) \leq \Vl(x,y) - 1 + b \1_{\Delta}(x,y) \eqsp, \quad \sup_{(x,y) \in \Delta} Q^{\ell-1} \Vl(x,y)  < \plusinfty \eqsp.
\end{equation}
\item an increasing sequence of integers $\{n_k,
k \in \nset \}$ and a concave function $\psi: \rset^+ \to \rset^+$ such that
$\lim_{v \to \plusinfty} \psi(v) = \plusinfty$ and
\begin{equation}
\label{eq:borne_suite_infini}
\sup_{k \in \nset} P^{n_k} [\psi \circ \Vl_{x_0}](x_0)   < \plusinfty \eqsp, \qquad P \Vl_{x_0} (x_0 ) < \plusinfty  \quad \text{for some $x_0 \in E$,}
\end{equation}
where $\Vl_{x_0}= \Vl(x_0,\cdot)$. 
\end{enumerate}
Then, $P$ admits a unique invariant distribution.
\end{theorem}
\begin{proof}
See \autoref{subsec:proof:theo:existence_pi}.
\end{proof}
If we now combine \autoref{hyp:small-set}($\Delta,\ell,\epsilon$) with a
condition which implies the control of the tail probabilities of the successive
return times to the coupling sets (more precisely, of the moments of order
larger than one of these return times) then the Wasserstein distance between
$P^n(x,\cdot)$ and $P^n(y,\cdot)$ may be shown to decrease at a subgeometric
rate.  To control these moments, it is quite usual to consider drift
conditions. In this paper, we focus on a class of drift conditions which has
been first introduced in \cite{douc:fort:moulines:soulier:2004}.
For $\Delta \in \B(E \times E)$, a function $\phi \in \mathbb{F}$,  a measurable function $V: E \rightarrow \coint{1, +\infty}$,
consider the following assumption:
\begin{assumption}[$\Delta,\phi,V$]
\label{hyp:drift_double}
\begin{enumerate}[(i)] 
\item There exists a constant $b < \infty$ such that for all $x,y \in E$:
\begin{equation}
\label{eq:drift_noyau_R2}
PV(x) + PV(y) \leq V(x) + V(y) - \phi(V(x)+V(y)) + b \1_{\Delta} (x,y) \eqsp.
\end{equation}
\item $\sup_{(x,y) \in \Delta} \{V(x) + V(y) \} < \plusinfty$.
\end{enumerate}
\end{assumption}
Not surprisingly, this condition implies that the return time to the coupling
set $\Delta$ possesses a first moment. This property combined with
\autoref{theo:existence-pi} yields
\begin{corollary}
\label{coro:existence_pi}
Assume that there exist a coupling kernel $Q$ for $P$, $\Delta \in \B(E \times E)$, $\ell \in \nset^*$, $\epsilon >0$, $\phi \in \mathbb{F}$ and $V : E \to \coint{1,\infty}$ such that \autoref{hyp:small-set}($\Delta,
\ell,\epsilon$)-\autoref{hyp:drift_double}($\Delta,\phi,V$) are satisfied.  Then,
$P$ admits a unique invariant probability measure $\pi$ and $\int_E \phi
  \circ V(x) \pi(\rmd x) < \infty$.
\end{corollary}
\begin{proof}
See \autoref{subsec:coro:existence_pi}.
\end{proof}
We now derive expressions of the rate of convergence and make explicit the dependence upon
the initial condition of the chain.
For $\phi \in \mathbb{F}$, set
\begin{equation}
\label{eq:defHphi}
H_\phi(t) = \int_1 ^t \frac{1}{\phi(s)} \dint s\eqsp.
\end{equation}
Since for $t \geq 1$, $\phi(t) \leq \phi(1) + \phi'(1)(t-1)$, the function
$H_\phi$ is monotone increasing to infinity, twice continuously differentiable
and concave. Its inverse, denoted $H_\phi^{\inv}$, is well defined on
$\rset_+$, is twice continuously differentiable and convex (see e.g.
\cite[Section 2.1]{douc:fort:moulines:soulier:2004}). 
\begin{theorem}
\label{theo:convergence_log_double_drift_1}
Assume that there exist a coupling kernel $Q$ for $P$, $\Delta \in \B(E \times E)$, $\ell \in \nset^*$, $\epsilon >0$, $\phi \in \mathbb{F}$ and $V : E \to \coint{1,\infty}$ such that \autoref{hyp:small-set}($\Delta,\ell,\epsilon$)-\autoref{hyp:drift_double}($\Delta,\phi,V$)
are satisfied.  Let $\pi$ be the invariant probability of $P$.
\begin{enumerate}[(i)]
\item \label{theo:item_rate1} There exist constants $\{C_i\}_{i=1}^3$ such that for all
  $x \in E$ and all $n \geq 1$
\begin{multline*}
  W_{d}(P^n(x,\cdot),\pi)
  \leq C_1 V(x) /H^{\inv}_{\phi}(n/2) + C_2 / \phi(H^{\inv}_{\phi}(n/2)) \\
  + C_3 / H^{\inv}_{\phi}(-\log(1-\epsilon) \, n/\{2( \log(H^{\inv}_{\phi}(n)) - \log(1-\epsilon)) \}) \eqsp.
\end{multline*}
\item \label{theo:item_rate2} For all $\delta \in \ooint{0,1}$, there exists a constant $C_\delta$ such
  that for all $x \in E$ and all $n \geq 1$ 
\begin{equation*}
  W_{d}(P^n(x,\cdot),\pi)
  \leq   C_\delta \, V(x)/ \phi (\{H^{\inv}_{\phi}(n)\}^\delta) \eqsp.
\end{equation*}
\end{enumerate}
The values of the constants $C_i$, for $i=1,2,3$, and $C_\delta$ are given explicitly in
the proof, and depend on $\Delta,  \ell,\epsilon,\phi,V,b$.
\end{theorem}
\begin{proof}
See \autoref{subsec:proof:theo:convergence_log_double_drift_1}
\end{proof}
We summarize in \autoref{tab:comparison} the rates of convergence obtained (for
a given $x \in E$) from \autoref{theo:convergence_log_double_drift_1} for usual
concave functions $\phi$: logarithmic rates $\phi(t) = (1+\log t)^\kappa$ for
some $\kappa >0$; polynomial rates $\phi(t) = t^\kappa$ for some $\kappa \in
\ooint{0,1}$; subexponential rates $\phi(t) = t/(1+\log t)^\kappa$ for some
$\kappa>0$.  Note that since $\phi \in \mathbb{F}$, the first
  term in the RHS of the bound in \eqref{theo:item_rate1} is not the leading
  term (for fixed $x$, when $n \to \infty$). In the case $\phi$ is logarithmic
  or polynomial, the leading term in the RHS is the second one so that the
  rate of decay is given by $1/\phi(H_\phi^\inv(n/2))$. For the logarithmic
and polynomial cases, the best rates are given by
\autoref{theo:convergence_log_double_drift_1}-\eqref{theo:item_rate1} and for
the subexponential case, by 
\autoref{theo:convergence_log_double_drift_1}-\eqref{theo:item_rate2}.
\begin{table}[!h]
\centering
\begin{tabular}{|c|c|c|c|}
  \hline
   Order of the rates  & $\phi(x) = (1+\log(x))^\kappa$
   & $\phi(x) = x^\kappa$
 & $\phi(x) = x/(1+\log(x))^\kappa$
\\

 of convergence in & for $\kappa > 0$
 & for $\kappa \in \ooint{0,1}$
  &  for $\kappa > 0$
   \\

&
&
& \\

 &  & set $\varsigma = \kappa/(1-\kappa)$
& set $\varsigma = 1/(1+\kappa)$
 \\
\hline
\autoref{theo:convergence_log_double_drift_1}
    &  $1/\log ^{\kappa} (n) $
    & $1/n^{\varsigma}$
      & $ \ \exp(-\delta((1+\kappa)n)^{\varsigma})$ \\
&&& for all $\delta \in \ooint{0,1}$
    \\
\hline
   \cite{douc:fort:moulines:soulier:2004}  & $ 1 / \log^{\kappa} (n)$
   & $1/n^{\varsigma}$
    & $n^{\kappa \varsigma} \exp(-((1+\kappa)n)^{\varsigma})$
    \\
\hline
\cite{butkovsky:2012} for all
$\delta \in \ooint{0,1}$
    & $ 1 / \log^{ \delta  \kappa} (n)$
   & $1/n^{\delta \varsigma}$
   &
   $\exists C > 0$
 \\
& & & $ \exp(-C n^{\varsigma})$   \\
\hline
\end{tabular}
\caption{\label{tab:comparison} Rates of convergence when $\phi$ increases at a logarithmic rate, a polynomial rate and a subexponential rate, obtained from \autoref{theo:convergence_log_double_drift_1} and from \cite[Theorem 2.1]{butkovsky:2012} and \cite[Section 2.3.]{douc:fort:moulines:soulier:2004}.}
\end{table}

In practice, it is often easier to establish a drift
inequality on $E$ rather than on $E \times E$ as in
\autoref{hyp:drift_double}($\Delta,\phi,V$). \autoref{prop:H1toDriftDouble}
relates the following single drift condition to the drift
\autoref{hyp:drift_double}. For a function $\phi \in\mathbb{F}$, a measurable function $V: E \rightarrow \coint{1, +\infty}$ and a
  constant $b \geq 0$, consider the following assumption
\begin{assumption}[$\phi,V,b$]
\label{hyp:drift_simple}
 $\phi(0) = 0$ and  for all $x \in E$,
\begin{equation}
\label{eq:drift}
PV(x)  \leq V(x) -\phi \circ V(x) + b \eqsp.
\end{equation}
\end{assumption}
\begin{theorem}
\label{prop:H1toDriftDouble}
Let $\phi \in\mathbb{F}$, a measurable function $V: E \rightarrow \coint{1, +\infty}$ and a
  constant $b \geq 0$ such that  \autoref{hyp:drift_simple}($\phi,V,b$) holds.  Then
\autoref{hyp:drift_double}($\{V \leq \upsilon\}^2, c \phi, V$) is satisfied for
any $\upsilon > \phi^\inv(2b)$ and with $c = 1 - 2b / \phi(\upsilon)$. 
\end{theorem}
The proof is postponed in \autoref{subsec:proof:prop:H1toDriftDouble}. Note
that we can assume without loss of generality that $t \mapsto \phi(t)$ is
concave increasing and continuously differentiable only for large $t$; see
\autoref{lem:prolongement_phi}.

\bigskip

  Our assumptions and results can be compared to \cite{butkovsky:2012} which
  also establish convergence in Wasserstein distance at a subgeometric rate
  under the single drift condition \autoref{hyp:drift_simple}($\phi,V,b$) and
  the following assumptions
\begin{list}{}{}
\item 
  \begin{enumerate}[B-(i)]
  \item \label{hyp:Bi} For all $x,y \in E$, $ W_d(P(x,\cdot),P(y,\cdot)) \leq d(x,y)$.
  \item \label{hyp:Bii} There exists $\eta >0$ such that the level set $\Delta
    = \{(x,y): V(x) + V(y) \leq \phi^{\inv}(2 b ) +\eta \}$ is $d$-small for
    $P$: there exists $\epsilon>0$ such that for any $x,y \in \Delta$,
    $W_d(P(x,\cdot), P(y,\cdot)) \leq (1-\epsilon) d(x,y)$.
  \end{enumerate}
\end{list}
Under these conditions, \cite[Theorem~2.1]{butkovsky:2012} shows the existence
and uniqueness of the stationary distribution $\pi$ and provides rates of
convergence to stationarity in the Wasserstein distance $W_d$; expressions for
these rates are provided in the last row of \autoref{tab:comparison} for
various choices of functions $\phi$.  It can be seen that our results always
improve the rates of convergence when compared to those
of~\cite{butkovsky:2012}

Let us compare the assumptions of \autoref{theo:convergence_log_double_drift_1}
to (B). It follows from \protect{\cite[Corollary 5.22]{VillaniTransport}} that
under B-\eqref{hyp:Bi} and B-\eqref{hyp:Bii}, there exists a coupling kernel
for $P$ (which is the coupling kernel realizing the lower bound in the
Monge-Kantorovitch duality theorem) such that \autoref{hyp:small-set}($\Delta,
1,\epsilon$) holds.  Since \autoref{prop:H1toDriftDouble} establishes that a
single drift condition of the form \autoref{hyp:drift_simple} implies a drift
condition of the form~\autoref{hyp:drift_double}, the assumptions of
\cite[Theorem~2.1]{butkovsky:2012} essentially differ from the assumptions of
\autoref{theo:convergence_log_double_drift_1} through the coupling set
assumption: \cite[Theorem 2.1]{butkovsky:2012} only covers coupling sets of
order $1$ when our result covers coupling sets of order $\ell$, for any $\ell
\geq 1$. This is an unnatural and sometimes annoying restriction since in
practical examples the order $\ell$ is most likely to be large (see e.g. the
examples in \autoref{sec:application}).  Note that the strategy consisting in
applying a result for a coupling set of order $1$ to the $\ell$-iterated kernel
is not equivalent to applying a result for a coupling set of order $\ell$ to
the one iterated kernel; we provide an illustration of this claim in
\autoref{subsec:SRWM-algorithm}. Checking
\autoref{hyp:small-set}($\Delta,\ell,\epsilon$) is easier than checking (B)
since allowing the coupling set to be of any order provides far more
flexibility.

Our results can also be compared to the explicit rates
in~\cite{douc:fort:moulines:soulier:2004} derived for convergence in total
variation distance.  In \cite{douc:fort:moulines:soulier:2004}, it is assumed
that $P$ is phi-irreducible, aperiodic, that the drift condition
\autoref{hyp:drift_simple} holds and that the level sets $\{V \leq \upsilon \}$
are small in the usual sense, \ie\ for some $\ell \in \N^*$, $\epsilon \in
\ooint{0,1}$ and a probability $\nu$ that may depend upon the level set,
$P^\ell(x,A) \geq \epsilon \nu(A)$, for all $x \in \{ V \leq \upsilon \}$ and
$A \in \B(E)$.  Under these assumptions, \cite[Proposition
2.5]{douc:fort:moulines:soulier:2004} shows that for any $x \in E$, $\lim_{n
  \to \infty} \phi(H_\phi^{\inv}(n)) \, W_{d_0}(P^n(x,\cdot),\pi) = 0$, where
$W_{d_0}$ is the total variation distance.  \autoref{tab:comparison} displays
the rate $r_\phi$ obtained in \cite{douc:fort:moulines:soulier:2004} (see
penultimate row) and the rates given by
\autoref{theo:convergence_log_double_drift_1} (row 2): our results coincide
with \cite{douc:fort:moulines:soulier:2004} for the polynomial and logarithmic
cases and the logarithm of the rate differs by a constant (which can be chosen
arbitrarily close to one in our case) in the subexponential case.
Nevertheless, we would like to stress that our conditions do not require
$\phi$-irreducibility and therefore apply in more general contexts.


%% file: application.tex
\section{Application}
\label{sec:application}
\subsection{A symmetric random walk Metropolis algorithm}
\label{subsec:SRWM-algorithm}
Let  $E \eqdef \{ k/4 ; k \in \zset \}$ endowed with the trivial distance $d_0$, thus $(E,d_0)$ is a Polish space.
 Consider a symmetric random walk Metropolis (SRWM) algorithm on $E$ for an  heavy tailed target distribution $\pi$  given by
  \begin{equation}
\label{eq:target_form_tails}
\pi(x) \propto 1/ (1+\abs{x})^{1+h} \eqsp, \quad \text{for all $x \in E$} \eqsp,
\end{equation}
where $h \in \ooint{0,1/2}$.  Starting at $x \in E$, the Metropolis algorithm
proposes at each iteration, a candidate $y$ from a random walk with a symmetric
increment distribution $q$ on $E$.  The move is accepted with probability $
\alpha(x,y) = 1 \wedge (\pi(y)/\pi(x))$. The Markov kernel associated with the SRWM
algorithm is given, for all $x \in E$ and $A \subset E $, by
\begin{equation*}
P(x, A ) = \sum_{y, x+y \in A} \alpha(x, x+y) q (y)  + \delta_x(A) \sum_{y \in E} (1-\alpha(x,x+y))q (y) \eqsp.
\end{equation*}
Assume that $q$ is the uniform distribution on $\{-1/4, 0 , 1/4 \}$. It is
easily checked that $P$ is irreducible and aperiodic.  In the following, we
prove that \cite[Theorem 2.1]{butkovsky:2012} cannot be applied to this case,
contrary to \autoref{theo:convergence_log_double_drift_1}.

We first prove that $P$ cannot be geometrically ergodic.  The proof essentially
follows from \cite[Theorem 2.2]{jarner:tweedie:2003}, where the authors
established necessary and sufficient conditions for the geometric and the
polynomial ergodicity of random walk type Markov chains on $\rset$.  

\begin{proposition}
  $P$ is not geometrically ergodic.
\end{proposition}
\begin{proof}
  The proof is by contradiction: we assume that $P$ is geometrically ergodic.
  Since it is also $P$ irreducible and aperiodic, the stationary distribution
  $\pi$ is unique and geometrically regular: for any set $A$ such that
  $\pi(A)>0$, there exists $L > 1$ such that $\E_\pi[L^{\tau_A}] = \sum_{x \in
    E} \pi(x) \E_x[L^{\tau_A}] < \infty$, where $\tau_A$ is the return time to
  $A$.  Choose $M > 0$, $A = \{ x \in E, |x| \leq M \}$. Since for $|x| \geq
  M$, $\tau_A \geq 4(|x|-M)$ $\mathbb{P}_x$-a.s. the regularity of $\pi$ claims
  that there exists $L >1$ such that $\sum_{x \in \zset} L^{|x|} \pi(x) <
  \infty$. This clearly yields to a contradiction.
\end{proof}
We then show that the Markov kernel $P$ satisfies a sub-geometric drift condition.
For $s \geq 0$, set $V_s(x) = 1 \vee \abs{x}^s$.
\begin{proposition} \label{prop:toyexample:drift}
For all $s \in \ooint{2,2+h}$, there exist $b,c >0$ such that for all $x \in E$
\begin{equation}\label{eq:drift:AR}
PV_s(x) \leq V_s(x) -c V_s(x)^{(s-2)/s} +b \eqsp.
\end{equation}
\end{proposition}

\begin{proof}
We have for all $x \geq 5/4$,
\begin{equation*}
PV_s(x) - V_s(x) = (x^{s}/3) \left( ((1-(4x)^{-1})^s-1) -(1-1/(5+4x))^{1+h}(1-(1+(4x)^{-1})^{s}) \right) \eqsp.
\end{equation*}
Since $(1-(4x)^{-1})^s-1 = -s/(4x)-(1-s)s/(32x^2) + o(x^{-2})$ and $(1-1/(5+4x))^{1+h} = 1-(1+h)/(4x) + (10+11 h +h^2)/(32x^2) +  o(x^{-2})$ as $x \to \plusinfty$, then $PV_s(x) - V_s(x) = x^{s-2}s(s-h-2)/48 + o(x^{s-2})$.
The same expansion remains valid as $x \to - \infty$ upon replacing $x$ by $-x$.
\end{proof}
Using this result, \cite[Proposition 2.5]{douc:fort:moulines:soulier:2004}
shows that for any $x \in E$, $P^n(x,\cdot)$ converges to $\pi$ in total
variation norm, at the rates $n^{\tilde{h}}$ for all $\tilde{h} \in
\ooint{0,h/2}$.

We can also apply \autoref{prop:H1toDriftDouble} and
\autoref{theo:convergence_log_double_drift_1}-\eqref{theo:item_rate1}. For any
$s \in (2,2+h)$,  \autoref{hyp:drift_simple}($\phi_s,V_s,b$) is satisfied with
$\phi_s(x)= c x^{(s-2)/s}$, $V_s(x)= 1 \vee |x|^s$ and $b < \plusinfty$.  For
$x,y \in E$ and $A,B \subset E$, consider the following kernel:
\[
Q((x,y),(A \times B )) = P(x,A  ) P(y, B) \1_{\{x \not = y\}} + P(x,A \cap B ) \1_{\{x=y\}} \eqsp.
\]
Clearly, $Q$ is a coupling kernel for $P$. Let us prove that for any $M>0$ and
any $\ell \geq 4M$, there exists $\epsilon>0$ such that
\autoref{hyp:small-set}$(\Delta, \ell, \epsilon)$ holds with $\Delta = \{|x|
\vee |y| \leq M \}$.  We have $Q d_0(x,y) \leq d_0(x,y)$ for every $x \ne y \in
E$ and by definition of $Q$, $Q d_0(x,x)= 0$ for every $x \in E$.  Let $M>0$,
$\ell \geq 4M$. For any $x,y \in \{|x| \vee |y| \leq M \}$ such that $|x|< |y|$
\begin{equation*}
\probaMarkovTilde{x,y}[X_{ \ell}= Y_{ \ell}] \geq \probaMarkovTilde{x,y}[X_{4 |y|} = Y_{4 |y|}]
\geq \probaMarkovTilde{x,y}[\tau_0^X=4|x|,X_{4|x|+1}=0,\dots,X_{4|y|}=0,\tau_0^Y=4|y|] \eqsp,
\end{equation*}
where $\tau_0^X = \inf \{ n \geq 1, X_n=0 \}$ and $\tau_0^Y= \inf \{ n \geq 1, Y_n=0\}$. Since \[
\text{$\probaMarkovTilde{x,y}[\tau_0^X=4|x|] \geq (1/3)^{4|x|}$, $\probaMarkovTilde{x,y}[\tau_0^Y=4|y|] \geq (1/3)^{4|y|}$,
$\probaMarkovTilde{x,y}[X_{4|x|+1}=0,\dots,X_{4|y|}=0] \geq (1/3)^{4(|y|-|x|)}$,}
\]
it follows that $Q^\ell d_0(x,y) = 1 - \probaMarkovTilde{x,y}[X_{ \ell}=
Y_{\ell}] \leq 1- (1/3)^{8|y|} \leq 1-(1/3)^{8M} d_0(x,y)$. This inequality
remains valid when $x=y$. This concludes the proof of
\autoref{hyp:small-set}$(\Delta, \ell, \epsilon)$. By
\autoref{prop:H1toDriftDouble}, the kernel $P$ is subgeometrically ergodic in
total variation distance at the rates $n^{\tilde{h}}$, for $\tilde{h} \in
\ooint{0,h/2}$.

In this example, \cite[Theorem 2.1]{butkovsky:2012} cannot be applied. Indeed,
on one hand, for any $M>0$ the set $\Delta_M = \{|x| \vee |y| \leq M \}$ is a
$(1, \epsilon,d_0)$-coupling set for $P^\ell$ iff $l \geq 4M$. This property is
a consequence of the above discussion (for the converse implication) and of the
equality $W_{d_0} (P^\ell(x,\cdot),P^{\ell}(y,\cdot))= 1$ if $|x-y| > \ell/2$
(for the direct implication). On the other hand, in order to check
B-\eqref{hyp:Bii} for some $\ell$-iterated kernel $P^\ell$, we have to prove
that there exists $\eta >0$ such that $\Delta_\star= \{(x,y) \in E^2 ; V_s(x) +
V_s(y) \leq (2b \ell/c)^{s/(s-2)} + \eta \}$ is a $(1,\epsilon,d_0)$-coupling
set for $P^\ell$ - the constants $b,c$ are given by
\autoref{prop:toyexample:drift}. Unfortunately, since $b/c \geq 1$ (apply the
drift inequality \ref{eq:drift:AR} with $x=0$), and $1/(s-2)\geq 2$, we get
 \[
 \{(x,y) \in E, |x| \vee |y| \leq 4 \ell^2\} \subset \{x,y \in E ; |x| \vee |y|
 \leq (2b\ell/c)^{1/(s-2)} \} \subset \Delta_\star \eqsp,
\]
Therefore whatever $\ell$, $\Delta_\star$ is not a $(1,\epsilon,d_0)$ coupling set
for $P^{\ell}$.

\subsection{Non linear autoregressive model}
In this section, we consider the functional autoregressive process
$\sequence{X}[n][\N]$ on $E=\R^p$, given by $X_{n+1} = g(X_n) + Z_{n+1}$
Denote by $\norm{\cdot}$ the Euclidean norm
on $E$ and $\boule{x}{M}$ the ball of radius $M \geq 0$ and
centered at $x \in \rset^p$, associated with this norm. Consider the following assumptions:
\begin{assumptionAR}
\label{hyp:autoreg_epsilon}
$\sequence{Z}[n][\N^*]$ is an independent and identically distributed (\iid)
zero-mean $\R^p$-valued sequence, independent of $X_0$, and satisfying $\int
\exp\parenthese{\beta_0 \norm{z}^{\kappa_0}} \mu(\dint z) < \plusinfty$, where
$\mu$ is the distribution of $Z_1$ for some $\beta_0 >0$ and $\kappa_0 \in
\ocint{0,1}$.
\end{assumptionAR}
\begin{assumptionAR}
\label{hyp:autoreg_g}
For all $M > 0$, $g: \R^p \rightarrow \R^p$ is  $C_M$-Lipschitz on $\boule{0}{M}$  with respect to $\norm{\cdot}$ where $C_M \in \ooint{0,1}$.
Furthermore, there exist positive constants $r, M_0$, and $\rho \in
\coint{0,2}$, such that $\norm{g(x)} \leq \norm{x}(1-r\norm{x}^{-\rho}) \quad \text{if } \norm{x} \geq M_0$.
\end{assumptionAR}
A simple example of function $g$ satisfying \autoref{hyp:autoreg_g} is $x
\mapsto x\cdot \max \parenthese{1/2, 1-1/\norm{x}^{\rho}}$ with $\rho \in
\coint{0,2}$. Denote by $P$ the Markov kernel defined by the process $(X_n)_n$.
\autoref{theo:drift_autoreg} establishes
\autoref{hyp:drift_simple}($\phi,V,b$) in the case where $\rho > \kappa_0$, and a geometric drift condition in the other case.
\begin{proposition}
\label{theo:drift_autoreg}
\cite[Theorem 3.3]{douc:fort:moulines:soulier:2004}
Assume \autoref{hyp:autoreg_epsilon} and \autoref{hyp:autoreg_g}.
\begin{enumerate}[(i)]
\item
\label{item:theo:drift_autoreg1}
If
$\rho > \kappa_0$,
there exist  $\beta \in \ooint{0, \beta_0}$ and $b,c > 0$ such that \autoref{hyp:drift_simple}($\phi,V,b$) holds with
 $ \phi(x) := c x (1+ \log(x))^{1 - \rho / (
    \kappa_0 \wedge (2 - \rho))}$ and $
  V(x) := \exp(\beta \norm{x}^{\kappa_0 \wedge (2 -
      \rho)})$.
\item
\label{item:theo:drift_autoreg2}
It $\rho \leq \kappa_0$, then there exist $b < \plusinfty$ and $\zeta \in
\ooint{0,1}$ such that for all $x \in \rset^p$, $PV(x) \leq \zeta V(x) + b$
where $V(x) = \exp(\beta \norm{x}^{\kappa_0})$ with $\beta \in
\ooint{0,\beta_0}$.
\end{enumerate}
\end{proposition}
\begin{proof}
The proof of \autoref{theo:drift_autoreg} is along the same lines as
\cite[Theorem 3.3]{douc:fort:moulines:soulier:2004} and is omitted \footnote{   We point out that in \cite{douc:fort:moulines:soulier:2004}, it is additionally required that the distribution of $Z_1$ has a nontrivial absolutely continuous
component which is bounded away from zero in a neighborhood of the origin. However, this condition is only required to establish the $\phi$-irreducibility of the Markov chain, which is not needed here.}
\end{proof}
Consider the coupling kernel $Q$ defined for all $x,y \in E$ and $A \in \B(E \times E)$ by
\begin{equation}
\label{eq:def_Q_auto_reg}
Q((x,y) , A) =  \int \1_A(g(x) + z,g(y) + z) \mu(\dint z) \eqsp.
\end{equation}
For $\eta > 0$, define $d_\eta(x,y) \eqdef 1 \wedge \eta^{-1} \norm{x-y}$.
\begin{proposition}
\label{lem:autoregressif_d_small}
Assume \autoref{hyp:autoreg_epsilon} and \autoref{hyp:autoreg_g}. For any $M>0$,
there exist $\epsilon,\eta >0$  such that $ \boule{0}{M} \times \boule{0}{M}$
is a $(1,\epsilon,d_\eta)$-coupling set.
\end{proposition}
\begin{proof}
Since $d_\eta(x,y) = \norm{x-y}/\eta$ for any $x, y \in \boule{0}{M}$ and $\eta= 2M$, we get under \autoref{hyp:autoreg_g},
\begin{equation}
\label{eq:alain}
 \E[\deta(g(x)+Z_1,g(y)+Z_1)] \leq  \eta^{-1} \norm{g(x)-g(y)} \wedge 1 \leq C_M \eta^{-1} \norm{x-y} \leq  C_M d_\eta(x,y) \eqsp.
\end{equation}
Finally, since \autoref{hyp:autoreg_g} implies that $g$ is $1$-Lipschitz on $\R^p$,  \eqref{eq:alain}
shows that $\E[\deta(g(x)+Z_1,g(y)+Z_1)] \leq \deta(x,y)$ for all $x,y \in \R^p$.
\end{proof}
For all $\eta, \eta' >0$, $d_{\eta}$ and $d_{\eta'}$ are Lipschitz equivalent,
\ie, there exists $C > 0$ such that for all $x,y \in \R^p$, $C^{-1}
d_{\eta}(x,y) \leq d_{\eta'}(x,y) \leq C d_{\eta}(x,y)$, which implies (see
\eqref{def_wasser}) that $W_{d_{\eta}}$ and $W_{d_{\eta'}}$ are Lipschitz
equivalent.
\begin{theorem}
\label{coro:rate_autoreg}
Assume \autoref{hyp:autoreg_epsilon} and \autoref{hyp:autoreg_g} hold. Then $P$ admits a unique invariant distribution $\pi$.
\begin{enumerate}[(i)]
\item
\label{item:coro:rate_autoreg1}
If
$\rho > \kappa_0$,
there exist two constants $C_1$ and $C_2$ such that for all $x \in \rset^p$ and $n
\in \N^*$
\[
W_{d_1}(P^n(x,\cdot),\pi) \leq C_1 V(x) \exp\parenthese{-C_2 n^{\varsigma}} \eqsp,
\]
where  $\varsigma = ( \kappa_0 \wedge (2 - \rho))/\rho$.
\item
\label{item:coro:rate_autoreg2}
If $\rho \leq \kappa_0$, then there exist $\tilde{\zeta} \in \ooint{0,1}$ and a
constant $C$ such that for all $x \in \rset^p$ and $n \in \nset^*$
\[
W_{d_1}(P^n(x,\cdot),\pi) \leq C V(x) \tilde{\zeta}^n \eqsp.
\]
\end{enumerate}
\end{theorem}

\begin{proof}
  By application of \autoref{coro:existence_pi},
  \autoref{theo:convergence_log_double_drift_1} and
  \autoref{prop:H1toDriftDouble}, we deduce \eqref{item:coro:rate_autoreg1}
  from \autoref{theo:drift_autoreg}-\eqref{item:theo:drift_autoreg1} and
  \autoref{lem:autoregressif_d_small}. By an application of \cite[Theorem 4.8,
  Corollary 4.11]{WeakHarris}, we deduce \eqref{item:coro:rate_autoreg2} from
  \autoref{theo:drift_autoreg}-\eqref{item:theo:drift_autoreg2} and
  \autoref{lem:autoregressif_d_small}.
\end{proof}
Perhaps surprisingly, we cannot relax the condition $\kappa_0 \in
\ocint{0,1}$, to obtain geometric convergence for $1 < \rho \leq \kappa_0$.
Indeed, \cite[Theorem~3.2(a)]{roberts:tweedie-Geom:1996} provides an example where
\autoref{hyp:autoreg_epsilon} and \autoref{hyp:autoreg_g} are satisfied
for $\kappa_0 =2$ and $\rho \in \ooint{1,2}$, but the chain fails to be geometrically ergodic (for the total variation distance).
\subsection{The preconditioned Crank-Nicolson algorithm}\label{sec:PCN}
In this section, we consider the preconditioned Crank-Nicolson algorithm
introduced in \cite{beskos:roberts:stuart:Voss:2008} and analyzed in
\cite{hairer:stuart:vollmer:2012} for sampling in a separable Hilbert space
$(\hilbert,\norm{\cdot})$ a distribution with density $\pi \propto \exp(-g)$
with respect to a zero-mean Gaussian measure $\gamma$ with covariance operator
$\covariance$; see \cite{bogachev:1998}.  This algorithm is studied in
\cite{hairer:stuart:vollmer:2012} under conditions which imply the geometric
convergence in Wasserstein distance.
\begin{algorithm}[!h]
 \DontPrintSemicolon
 \KwData{$\rho \in \coint{0,1}$}
 \KwResult{$\sequence{X}[n][\N]$}
 \Begin{
   Initialize $X_0$\; \For{$n \geq 0$ }{ Generate $Z_{n+1} \sim \gamma$.  \;
     Generate $U_{n+1} \sim \mathcal{U}(\ccint{0,1})$ \; \eIf{ $U_{n+1} \leq
       \alpha(X_n, \rho X_n + \sqrt{1-\rho^2} Z_{n+1})= 1 \wedge \exp(g(X_n) -
       g(\rho X_n + \sqrt{1-\rho^2} Z_{n+1}))$} { $X_{n+1} = \rho X_n +
       \sqrt{1-\rho^2} Z_{n+1}$ \;} {$X_{n+1} = X_n$ \;} } }
 \caption{Preconditioned Crank-Nicolson Algorithm}
\label{algo:pCN}
\end{algorithm}
We consider the convergence of the Crank-Nicolson algorithm under the weaker
condition \autoref{hyp:pCN} below for which the results in
\cite{hairer:stuart:vollmer:2012} cannot be applied. We will show that
subgeometric convergence can nevertheless be obtained.
\begin{assumptionCN}
\label{hyp:pCN}
 The function $g:\hilbert  \to \R$ is $\beta$-Hölder for some
  $\beta \in \ocint{0,1}$ \ie, there exists $C_g$, such that for all $x,y \in
  \hilbert$, $\abs{g(x)-g(y)} \leq C_g \norm{x-y}^\beta$.
\end{assumptionCN}
Examples of densities satisfying \autoref{hyp:pCN} are $g(x) = -\norm{x}^\beta$
with $\beta \in \ocint{0,1}$. The following Theorem implies that under
\autoref{hyp:pCN}, $\exp(-g)$ is $\gamma$-integrable (see
\protect{\cite[Theorem~2.8.5]{bogachev:1998}}).
\begin{theorem}[Fernique's theorem]
\label{theo:Fernique}
 There exist $\theta \in \R_+^*$ and  a constant $C_\theta$ such that
$\int_{\hilbert} \exp(\theta \norm{\xi}^2) \dint \gamma (\xi) \leq C_\theta$.
\end{theorem}
The Crank-Nicolson kernel $P_{\pCN}$  has been shown to be geometrically ergodic  by \cite{hairer:stuart:vollmer:2012} under the assumptions that $g$ is globally Lipschitz and that there exist positive
constants $C, {M}_1, {M}_2$ such that for $x \in \hilbert$ with $\norm{x} \geq M_1$, $\inf_{z \in \boulefermee{\rho x}{{M}_2}} \exp( g(x) - g(z) ) \geq C$ (see \cite[Assumption 2.10-2.11]{hairer:stuart:vollmer:2012}), where
we denote by $\boule{x}{M}$ the open ball centered at $x \in \hilbert$
and of radius $M >0$ associated with $\norm{\cdot}$, and by $\boulefermee{x}{M}$ its closure.  Such an
assumption implies that the acceptance ratio $\alpha(x,\rho x +
\sqrt{1-\rho^2} \xi)$ is bounded from below as $\norm{x} \to \infty$
uniformly on $\xi \in \boulefermee{0}{{M}_2/\sqrt{1-\rho^2}}$.  In \autoref{hyp:pCN}, this condition is weakened
in order to address situations in which the acceptance-rejection ratio
vanishes when $\norm{x} \to \infty$: this happens when $\lim_{\norm{x}
  \to \plusinfty} \{ g(\rho x ) - g(x) \} = \plusinfty$. We first check that  \autoref{hyp:drift_simple}$(\phi,V,b$) is satisfied with
\begin{equation}
\label{eq:lyap_fun_pCN}
V(x) = \exp(s \norm{x}^2) \eqsp,
\end{equation}
where $s= (1-\rho)^2 \theta /16$ and $\theta$ is given by \autoref{theo:Fernique}.
\begin{proposition}
\label{lem:drift_g_hold}
Assume \autoref{hyp:pCN}, and let $\rho \in \coint{0,1}$. Then there exist $b
\in \R_+$ and $c \in \ooint{0,1}$ such that for all $x \in \hilbert$
\[
P_{\pCN}V(x) \leq V(x) - \, {\phi}\circ V(x)
+b  \eqsp,
\]
where $\phi \in \mathbb{F}$ and ${\phi}(t) \sim_{t \to
  \infty} c t \exp(- \{\log(t)/\kappa\}^{\beta/2})$, with $\kappa = \theta
C_g^{-2/\beta}/36 $.
\end{proposition}

\begin{proof}
The proof is postponed to \autoref{sec:proof:lem:drift_g_hold}.
\end{proof}
We now deal with showing \autoref{hyp:small-set}.  To that goal, we introduce
the distance $d_\eta(x,y) = 1 \wedge \eta^{-1} \norm{x-y}^\beta$, for any $\eta
>0$, and for $x,y \in E$ the basic coupling $Q_{\pCN}$ between
$P_{\pCN}(x,\cdot)$ and $P_{\pCN}(y,\cdot)$: the same Gaussian variable $\Xi$
and the same uniform variable $U$ are generated to build $X_1$ and $Y_1$, with
initial conditions $x,y$.  Define $\Lambda_{(x,y)}(z) = (\rho x +
\sqrt{1-\rho^2}z , \rho y + \sqrt{1-\rho^2}z)$ and $\tildegamma_{(x,y)}$ the
pushforward of $\gamma$ by $\Lambda_{(x,y)}$. Then an explicit form of
$Q_{\pCN}$ is given, for $A \in \B(\hilbert \times \hilbert)$, by:
\begin{multline}
  \label{eq:def_Q_pCN}
  Q_{\pCN}((x,y),A) = \int_{A} \alpha(x,v)\wedge \alpha(y,t)
  \dint \tildegamma_{(x,y)}(v,t) +\int_{\hilbert \times \hilbert} \parenthese{\alpha(y,t)-\alpha(x,v) }_+ \1_{A}(x,t)  \dint \tildegamma_{(x,y)}(v,t) \\
  + \int_{\hilbert \times \hilbert  } \parenthese{\alpha(x,v)-\alpha(y,t)}_+   \1_{A}(v,y) \dint \tildegamma_{(x,y)}(v,t)+ \delta_{(x,y)}(A) \int_{\hilbert \times \hilbert} (1-\alpha(x,v) \vee \alpha(y,t)) \dint
  \tildegamma_{(x,y)}(v,t)
\end{multline}
  where for $u \in \R$, $(u)_+ = \max(u,0)$.
  The following Proposition shows that \autoref{hyp:small-set} is satisfied.
\begin{proposition}
\label{lem:g_simple_smallness_hold}
Assume \autoref{hyp:pCN}.  There exists $\eta > 0$ such that, $Q_{\pCN}$ is a
$d_\eta$-weak contraction and for every $u > 1$, there exist $\ell \geq 1$ and
$\epsilon >0$ such that $\{V \leq u \}^2$ is a $(\ell, \epsilon,
d_\eta)$-coupling set.
\end{proposition}
\begin{proof}
See \autoref{proof:lem:g_simple_smallness_hold}
\end{proof}
Note that for all $\eta >0$, $d_\eta$ is  Lipschitz equivalent to $d_1$, therefore $W_{d_\eta}$ and $W_{d_1}$ are Lipschitz equivalent.
As a consequence of \autoref{lem:drift_g_hold}, \autoref{lem:g_simple_smallness_hold},
\autoref{theo:convergence_log_double_drift_1} and \autoref{prop:H1toDriftDouble},
we have
\begin{theorem}
\label{theo:pCN}
Let $P_{\pCN}$ be the kernel of the preconditioned Crank-Nicolson algorithm
with target density $\dint \pi \propto \exp(-g) \dint \gamma$ and design
parameter $\rho \in \coint{0,1}$. Assume \autoref{hyp:pCN}.  Then $P_{\pCN}$
admits $\pi$ as a unique invariant probability measure and there exist $C_1,
C_2$ such that for all $n \in \N^*$ and $x \in \hilbert$
\[
W_{d_1}(P_{\pCN}^n(x,\cdot),\pi )
\leq C_1 V(x) \exp \parenthese{-\kappa (\log(n)- C_2 \log (\log(n)))^{2/\beta}} \eqsp,
\]
where $V$ is given by \eqref{eq:lyap_fun_pCN}, $d_1(x,y) = \norm{x-y}^\beta
\wedge 1$ and $\kappa = \theta C_g^{-2/\beta}/36$ for $\theta$ given by
\autoref{theo:Fernique}.
\end{theorem}
\autoref{theo:pCN} covers the case of the independant sampler (case $\rho=0$).
Both the rate of convergence through the constant $C_2$ and the control in the
initial value $x$ throught $C_1$ and the function $V$ depends on $\rho$.


%% file: section_proof.tex
\section{Proofs of \autoref{sec:main_results}}
\label{sec:proof}
In this section, $C$ is a constant which may take different values upon each appearance.

For $\Delta \in \B(E \times E)$, $\ell \in \nset^*$ and a canonical Markov
chain on the space $( (E \times E)^\N, (\B(E) \otimes \B(E))^{\otimes \N})$,
denote by $T_0 = \inf \defEns{n \geq \ell, (X_n,Y_n ) \in \Delta}$ the first
return time to $\Delta$ after $\ell-1$ steps.  Then, define recursively for $j
\geq 1$,
\begin{equation}
\label{eq:def_T_m}
T_j = T_0 \circ \theta ^{T_{j-1}} + T_{j-1} =
T_0  + \sum_{k=0}^{j-1}  T_0 \circ \theta^{T_k} \eqsp,
\end{equation}
where $\theta$ is the shift operator.

Let $Q$ be a coupling kernel for $P$. Hereafter, $\{(X_n, Y_n), n \in \N \}$ is
the canonical Markov chain on the space $( (E \times E)^\N, (\B(E) \otimes
\B(E))^{\otimes \N})$ with Markov kernel $Q$. We denote by
$\widetilde{\mathbb{P}}_{x,y}$ and $\widetilde{\mathbb{E}}_{x,y}$ the
associated canonical probability and expectation, respectively, when the
initial distribution of the Markov chain is the Dirac mass at $(x,y)$.

For any $n \in \N^\star$ and $x,y \in E$, the $n$-iterated kernel $Q^n((x,y),
\cdot)$ is a coupling of $(P^n(x,\cdot), P^n(y,\cdot))$; hence
$W_d(P^n(x,\cdot),P^n(y,\cdot)) \leq \expeMarkovTilde{x,y}{d(X_n,Y_n)}$.
Define the filtration $\{ \filtrationTilde_n, n \geq 0\}$ by
$\filtrationTilde_n= \sigma( (X_k,Y_k), k \leq n)$.

Before proceeding to the actual derivation of the proofs, we present a roadmap of them.
The key step for our results is given by the following inequality: for any $x,y \in E$ and $n,m \geq 1$,
\begin{equation}
  \label{eq:sketchproof:objectif}
  W_d\left( P^n(x,\cdot), P^m(y,\cdot) \right) \leq B(n,m) \ \left(V(x) + V(y) \right) \eqsp,
\end{equation}
with $\lim_{n,m \to \plusinfty} B(n,m) = 0$. Under the assumptions of
\autoref{theo:existence-pi}, this inequality will imply that $P$ admits at most
one invariant probability.  In addition, by applying
\eqref{eq:sketchproof:objectif} with $n \leftarrow n+m$, and $y \leftarrow x$,
we show that $\{P^n(x,\cdot), n \in \N \}$ is a Cauchy sequence in
$(\Pens(E),W_d)$ and therefore converges in $W_d$ to some probability measure
$\pi_x$ which is shown to be invariant for $P$. Since $P$ admits one invariant
probability measure, then $\pi_x$ does not depend on $x$ (see
\autoref{subsec:proof:theo:existence_pi}). The proof of
\autoref{coro:existence_pi} consists in verifying that the assumptions of
\autoref{theo:existence-pi} are satisfied. 

 The proof of \autoref{theo:convergence_log_double_drift_1} also follows
  from \eqref{eq:sketchproof:objectif}, but an explicit expression of $B$ is
  required (see \autoref{lem:itere_couplage}). Taking $n=m$ and integrating
  this inequality \wrt\ the unique invariant distribution $\pi$ will conclude
  the proof.

Let us now explain the computation of the upper bound
(\ref{eq:sketchproof:objectif}).  
The contraction property of $Q$ (see \autoref{hyp:small-set}\eqref{item:contract_coupling_2})
combined with the Markov property of $\defEns{ (X_n, Y_n), n \in \N }$ imply
that $\{ d(X_n,Y_n), n \in \N \}$ is a supermartingale with respect to the
filtration $\widetilde{\filtration}_n$; this property yields $
\expeMarkovTilde{x,y}{d(X_n,Y_n)} \leq (1-\epsilon)^{m-1} +
\probaMarkovTilde{x,y}[T_m \geq n] $ for any $n,m \geq 0$. By the Markov
inequality, for any increasing rate function $R$, it holds
\begin{equation}
  \label{eq:sketchproof:ineqmarkov}
  \expeMarkovTilde{x,y}{d(X_n,Y_n)} \leq (1-\epsilon)^{m-1} +
\frac{\expeMarkovTilde{x,y}{R(T_m)}}{R(n)} \eqsp.
\end{equation}
The last step of the proof is to compute an upper bound for the moment
$\expeMarkovTilde{x,y}{R(T_m)}$. Then $m$ is chosen in order to balance the two
terms in the RHS of \eqref{eq:sketchproof:ineqmarkov}.

We preface the proof of our results by the following result.

\begin{proposition}
\label{lem:surmartingale2}
Assume that there exists a coupling kernel $Q$ for $P$, $\Delta \in \B(E \times
E)$, $\ell \in \nset^*$ and $\epsilon>0$ such that
\autoref{hyp:small-set}($\Delta,\ell,\epsilon$) holds. Then, for all $x,y \in
E$, and $n \geq 0$, $m\geq 0$ :
\begin{equation*}
\expeMarkovTilde{x,y}{d(X_n,Y_n)} \leq  (1-\epsilon)^{m} + \probaMarkovTilde{x,y}[T_m \geq  n] \eqsp.
\end{equation*}
\end{proposition}

\begin{proof}
  Set $Z_n = d(X_n,Y_n)$; under
  \autoref{hyp:small-set}($\Delta,\ell,\epsilon$),
  $\{(Z_n,\filtrationTilde_n)\}_{n \geq 0}$ is a bounded non-negative
  supermartingale and for all $(x,y) \in \Delta$,
  $\expeMarkovTilde{x,y}{Z_\ell} \leq (1-\epsilon)d(x,y)$. Denote by $Z_\infty$
  its $\probaMarkovTilde{x,y}$-\as\ limit.  By the optional stopping theorem,
  we have for every $m \geq 0$: $ \expeMarkovTilde{x,y}{Z_{T_{m+1}}
    \sachant{\filtrationTilde_{T_{m}+\ell}}} \leq Z_{T_{m}+\ell}$.  On the
  other hand, by the strong Markov property, $\expeMarkovTilde{x,y}{Z_{T_{m}+\ell}
    \sachant{\filtrationTilde_{T_{m}}}} \leq (1-\epsilon) Z_{T_m}$.  By
  combining these two relations, we get: $\expeMarkovTilde{x,y}{Z_{T_{m+1}}
    \sachant{\filtrationTilde_{T_{m}}}} \leq (1-\epsilon) Z_{T_{m}}$.  Since
  $Z_n$ is upper bounded by $1$, the proof follows from \cite[lemma
  3.1]{Jarner_Wasserstein}.
\end{proof}

\subsection{Proof of \autoref{theo:existence-pi}}
\label{subsec:proof:theo:existence_pi}
By \autoref{lem:surmartingale2} and the Markov inequality for all $m \geq 0$,
we get
\begin{equation}
\label{eq:laborne}
\expeMarkovTilde{x,y}{d(X_n,Y_n)} \leq (1-\epsilon)^{m} + n^{-1} \expeMarkovTilde{x,y}{T_m} \eqsp.
\end{equation}
Using \eqref{eq:def_T_m} and  the strong Markov property, we obtain
$\expeMarkovTilde{x,y}{T_m}= \expeMarkovTilde{x,y}{T_0} + \expeMarkovTilde{x,y}{\sum_{k=0}^{m-1} \expeMarkovTilde{X_{T_k},Y_{T_k}}{T_0}} $.
Using \cite[Proposition~11.3.3]{bible} and the Markov property we have that
\[
\expeMarkovTilde{x,y}{T_0}  \leq   Q^{\ell-1} \Vl(x,y) + b + \ell - 1 \eqsp,
\]
which implies that $\expeMarkovTilde{x,y}{T_m} \leq m \sup_{(x,y)\in \Delta}
Q^{\ell-1} \Vl(x,y) + Q^{\ell-1} \Vl(x,y) + (m+1) \left( b + \ell -1\right)$,
where the constant $b$ is defined in \eqref{eq:simple-drift-condition}.
Plugging this inequality into \eqref{eq:laborne} and taking $m = \lceil -
\log(n)/ \log(1-\epsilon) \rceil$ implies that there exists $C < \infty$
satisfying
\begin{equation}
\label{eq:bornealain}
Q^n d(x,y)= \expeMarkovTilde{x,y}{d(X_n,Y_n)} \leq  C (\log(n)/n) Q^{\ell-1} \Vl(x,y) \leq  C (\log(n)/n) \Vl(x,y)    \eqsp,
\end{equation}
where we have used that $Q^{\ell-1} \Vl(x,y) \leq \Vl(x,y) + b (\ell-1)$ (the constant $C$ takes different values upon each appearance).

\subsubsection*{Uniqueness of the invariant probability}
The proof is by contradiction.  Assume that there exist two invariant
distributions $\pi$ and $\nu$, and let $\lambda \in \couplage{\pi}{\nu}$.
According to
\autoref{lem:contract_coupling_weak_contraction}-\eqref{lem:contraction_rev_mu_nu_2},
we have for every integer $n$,
\[
W_{d}(\pi,\nu) = W_{d}(\pi P^n , \nu P^n) \leq \int_{E\times E} Q^nd(x,y)
\, \lambda (\dint x, \dint y) \eqsp.
\]
We prove that the RHS converges to zero by application of the dominated
convergence theorem.  It follows from \eqref{eq:bornealain} that for all $x,y
\in E$ and $n \geq 0$, $g_n(x,y) \eqdef Q^nd(x,y) \leq C \Vl (x,y) \;
\log(n)/n$ for some $C < \infty$.  Therefore, the sequence of functions
$\sequence{g}[n][\N]$ converges pointwise to $0$. Since $d \leq 1$, $g_n(x,y)
\leq 1$. Hence, by the Lebesgue theorem, $\int_{E\times E} g_n(x,y) \, \lambda
(\dint x,\dint y) \flecheLimite 0$ showing that $W_{d}(\pi,\nu) = 0$, or
equivalently $\nu = \pi$ since $W_d$ is a distance on $\Pens(E)$.

\subsubsection*{Existence of an invariant measure}
  Let $x_0 \in E$. We first show that there exists $\defEns{m_k, k \in \N}$ such that
  $\{P^{m_k}(x_0, \cdot), k \in \N \}$ is a Cauchy sequence for $W_{d}$.
Let $n,k \in \N^*$ and choose $M \geq 1$.  By
\autoref{lem:contract_coupling_weak_contraction}-\eqref{lem:contraction_rev_mu_nu_2}:
\begin{multline}
W_{d}(P^n(x_0,\cdot), P^{n+n_k}(x_0,\cdot))
\leq \inf_{\lambda \in ~ \couplage{\delta_{x_0}}{
    P^{n_k}(x_0,\cdot) }} \left\lbrace \int_{E\times E}
    \1_{\{\Vl(z,t) \geq  M\}} Q^n \, d(z,t)  \lambda (\dint z, \dint t)  + \right. \\
\left.    \int_{E\times E} \1_{\{\Vl(z,t)<  M\}} Q^nd(z,t) \lambda (\dint z, \dint t) \right\rbrace \eqsp.
\label{eq:existence_base}
\end{multline}
We consider separately the two terms.  Set $M_\psi = \sup_k P^{n_k}[\psi \circ
\Vl_{x_0}](x_0)$. Let $\lambda \in \couplage{\delta_{x_0}}{ P^{n_k}(x_0,\cdot)
}$. Since $d$ is bounded by $1$, we get
\begin{multline}
\label{eq:existence_premiere_borne}
\int_{E\times E} \1_{\{\Vl(z,t) \geq M\}} Q^nd(z,t)
\lambda (\dint z, \dint t) \leq \int_{E\times E} \1_{\{\Vl(z,t) \geq M\}}
\lambda (\dint z, \dint t) \leq P^{n_k}(x_0, \{\Vl_{x_0} \geq M\}) \\ \leq
P^{n_k}(x_0, \{ \psi \circ \Vl_{x_0} \geq \psi(M)\}) \leq P^{n_k}[\psi \circ
  \Vl_{x_0}](x_0)/\psi(M) \leq  M_\psi/\psi(M) \eqsp,
\end{multline}
where we have used \eqref{eq:borne_suite_infini} and the Markov inequality. In addition by
\eqref{eq:bornealain}, there exists $C>0$ such that:
\begin{equation*}
\int_{E\times E} \1_{\{\Vl(z,t) < M\}}
 Q^nd(z,t) \lambda (\dint z, \dint t)
   \leq C (\log(n)/n) \, \int_{E\times E} \1_{\{\Vl(z,t) < M\}}  \Vl(z,t)  \lambda (\dint z, \dint t)
  \eqsp.
\end{equation*}
Furthermore, $ x \mapsto \psi(x)/x$ is non-increasing so that $\Vl(z,t) \leq M
\psi(\Vl(z,t))/\psi(M)$ on $\defEns{\Vl(z,t) \leq M}$. This inequality and
\eqref{eq:borne_suite_infini} imply
\begin{equation} \label{eq:deuxieme_terme_ineg_propo_finale_existence}
\int_{E\times E} \1_{\{\Vl(z,t) < M\}} Q^nd(z,t)  \lambda (\dint z, \dint t)
\leq C(\log(n) /n )   M_\psi M / \psi(M)   \eqsp.
\end{equation}
Plugging \eqref{eq:existence_premiere_borne} and
\eqref{eq:deuxieme_terme_ineg_propo_finale_existence} in
\eqref{eq:existence_base}, we have for every $M >0$, $n,k \in \N^*$
\[
W_{d}(P^n(x_0,\cdot) , P^{n+n_k}(x_0,\cdot)) \leq \frac{M_\psi}{ \psi(M)} +
C(\log(n)/n)\parenthese{ M_\psi M / \psi(M)} \eqsp.
\]
Setting $M = n/ \log(n)$, we get that for all $n,k \in \N^*$
\begin{equation}
\label{eq:base_construction_Cauchy}
W_{d}(P^n(x_0,\cdot) , P^{n+n_k}(x_0,\cdot)) \leq C /\psi(n/ \log(n))\eqsp.
\end{equation}
Since $\lim_{x \to \plusinfty} \psi(x) =
\plusinfty$ and $\lim_{k \to \plusinfty } n_k = \plusinfty$
there exists $\{u_k, k \in \N \}$ such that $u_0 =1$ and for $k \geq 1$, $u_{k} = \inf \{ n_l \ | \ l \in \N ; \psi
(n_l/\log(n_l))  \geq 2^{k} \} $.
Set $m_k = \sum_{i=0}^k u_i$.  Since for all $k \in \N$, $m_{k+1} = m_k +
u_{k+1}$, by \eqref{eq:base_construction_Cauchy}, $W_{d}(
P^{m_k}(x_0,\cdot), P^{m_{k+1}}(x_0,\cdot)) \leq C 2^{-k}$, which
implies that the series $\sum_k W_{d}( P^{m_k} (x_0,\cdot),
P^{m_{k+1}}(x_0,\cdot))$ converges and $\{ P^{m_k}(x_0,\cdot) , k \in \N \}$ is a Cauchy sequence in $(\Pens(E),W_{d})$.

Since $(\Pens(E), W_{d})$ is Polish, there
exists $\pi \in \Pens(E)$ such that $\lim_{k \to \plusinfty}
W_{d}(P^{m_k}(x_0,\cdot),\pi) = 0$.  The second step is to prove that $\pi$ is
invariant. Since $\lim_{k \to \plusinfty} W_d(P^{m_k}(x_0,\cdot),\pi ) = 0$, by the triangular inequality
it holds
\begin{align}  \label{eq:preuve_invariance_base}
  W_{d}(\pi, \pi P) \leq \lim_{k \to \plusinfty} W_{d}(P^{m_k}(x_0,\cdot),
  \delta_{x_0}PP^{m_k}) + \lim_{k \to \plusinfty}
  W_{d}(\delta_{x_0}P^{m_k}P,\pi P) \eqsp.
\end{align}
By  \autoref{lem:contract_coupling_weak_contraction}-\eqref{lem:contraction_rev_mu_nu_2} and \eqref{eq:bornealain} , there exists $C$ such that for any $k \geq 1$,
\begin{multline*}
W_{d}(P^{m_k}(x_0,\cdot), \delta_{x_0} P^{m_k+1}) \leq \inf_{ \lambda \in
    \couplage{\delta_{x_0}}{\delta_{x_0}P}} \int_{E \times E} Q^{m_k}d(z,t)\dint \lambda(z,t)\\
   \leq C( \log(m_k) /m_k )  \inf_{ \lambda \in  \couplage{\delta_{x_0}}{\delta_{x_0}P}} \int_{E \times E} \Vl(z,t) \lambda(\rmd z , \rmd t) \leq C(\log(m_k)/m_k)  P \Vl_{x_0} (x_0)  \eqsp.
\end{multline*}
 By definition, $ \lim_k m_k = + \infty$ so that by
\eqref{eq:borne_suite_infini}, the RHS converges to $0$ when $k
\to \plusinfty$. In addition, by
\autoref{lem:contract_coupling_weak_contraction}-\eqref{lem:contract_rev_3},
$W_{d}(\delta_{x_0}P^{m_k}P,\pi P) \leq W_{d}(P^{m_k}(x_0,\cdot),\pi )$, and
this RHS converges to $0$ by definition of $\pi$. Plugging these results in
\eqref{eq:preuve_invariance_base} yields $ W_{d}(\pi,\pi P ) = 0$, and
therefore $\pi P = \pi$.

\subsection{Proof of \autoref{coro:existence_pi}}
\label{subsec:coro:existence_pi}
We prove that the assumptions of \autoref{theo:existence-pi} are satisfied. Set
$\Vl(x,y) = 1+(V(x) +V(y))/\phi(2)$.  Since $Q$ is a coupling for $P$, it holds
\[
Q \Vl(x,y) = 1 + (1/\phi(2) )\left(PV(x) + PV(y) \right) \leq \Vl(x,y)
- \frac{\phi(V(x) +V(y))}{\phi(2)} + (b/\phi(2)) \1_\Delta(x,y) \eqsp.
\]
This yields the drift inequality (\ref{eq:simple-drift-condition}) upon noting
that $\phi$ is increasing and $V \geq 1$ so that $\phi(V(x) +V(y)) /\phi(2)
\geq 1$.  By iterating this inequality, we have for any $\ell$,
\[
\sup_{(x,y) \in \Delta} \{Q^{\ell-1} \Vl(x,y)\} \leq  \sup_{(x,y) \in \Delta}
\{\Vl(x,y) \}  + b (\ell-1)/\phi(2) \eqsp,
\]
and the RHS is finite since by assumption, $\sup_{(x,y) \in \Delta} \{V(x)
+V(y)\} < \infty$.

Let $x_0 \in E$. Under \autoref{hyp:drift_double}($\Delta,\phi,V)$, $PV(x) \leq PV(x) + PV(x_0)
\leq V(x) - \phi \circ V(x) + b +V(x_0)$ where we have used that $\phi (V(x) +
V(x_0)) \geq \phi(V(x))$. This implies that for every $n \in \N^*$, $n^{-1}
\sum_{k=0}^{n-1} P^k (\phi \circ V)(x) \leq b + V(x_0)$ $+V(x)/n$.
For any $x$, we have $P\Vl_{x}(x) < \infty$.  Finally, since $\phi \in
\mathbb{F}$, we can set $\psi = \phi$. Let us define the increasing sequence
$\{n_k, k \in \N\}$.  Set $M_\phi > b+V(x_0)$; there exists an increasing
sequence $\{n_k , k \in \N \}$ such that $\lim_k n_k = + \infty$ and
\begin{equation}
\label{eq:def_nk}
P^{n_k}(\phiV) (x_0) \leq M_\phi \eqsp, \text{ for all } k \in \N \eqsp.
\end{equation}
Finally, \cite[lemma~4.1]{butkovsky:2012} implies  $\int_E \phi \circ V(x)
    \pi(\rmd x) < \infty$.

\subsection{Proof of \autoref{theo:convergence_log_double_drift_1}}
\label{subsec:proof:theo:convergence_log_double_drift_1}
We preface the proof by some preliminary technical results. By using
\autoref{lem:surmartingale2}, for every $x,y \in E$ and $m \geq 0$,
$\expeMarkovTilde{x,y}{d(X_n,Y_n)} \leq (1-\epsilon)^{m} +
\probaMarkovTilde{x,y}[T_m > n]$.  The crux of the proof is to obtain estimates
of tails of the successive return times to $\Delta$.  Following
\cite{Tuominen_subgeo}, we start by considering a sequence of drift conditions
on the product space $E \times E$. For $\Delta \in \B(E \times E)$, $\ell \in
\nset^*$, a sequence of measurable functions $\defEns{\Vl_n, n \in \N }$,
$\Vl_n : E\times E \rightarrow \R_+$, a function $r \in \Lambda$ and a
constant $b < \infty$, let us consider the following assumption:
\begin{assumption*}($\Delta,\ell,\Vl_n,r,b$)
\label{hyp:suite_drift}
For all $x,y \in E$~:
\begin{equation*}
\label{eq:sequence-drift}
Q \Vl_{n+1}(x,y)    \leq \Vl_n(x,y) -r(n) + b r(n) \1_{\Delta} (x,y)\eqsp, \quad \text{and} \quad
\quad \sup_{(x,y) \in \Delta} Q^{\ell-1} \Vl_0(x,y) < \infty \eqsp.
\end{equation*}
\end{assumption*}
Under \textbf{A}($\Delta,\ell,\Vl_n,r,b$), we first  obtain bounds on the moments
$\expeMarkovTilde{x,y}{R(T_0)}$ for $x,y \in E$ (see
\autoref{propo:temps_darret_fini}), where
\begin{equation}
  \label{eq:definition:R}
  R(t) = 1 + \int_0^t r(s) \rmd s \eqsp, t \geq 0\eqsp.
\end{equation}
We will then deduce bounds for $\probaMarkovTilde{x,y}[T_m \geq n]$ (see
\autoref{lem:surmartingale_contraction}). Set
\begin{equation}
\label{eq:definition-c-r}
c_{1,r}= \sup_{k \in \nset^*} R(k) / \sum_{i=0}^{k-1} r(i)  \eqsp,  \quad c_{2,r} = \sup_{m,n \in \N} R(m+n)/\{R(m) R(n) \} \eqsp.
\end{equation}
It follows from \autoref{lem:suite_sous_geometrique} that these constants are
finite.

\begin{proposition}
\label{propo:temps_darret_fini}
Assume that there exist a coupling kernel $Q$ for $P$, $\Delta \in \B(E\times
E)$, $\ell \in \nset^*$, a sequence of measurable functions $\defEns{\Vl_n, n
  \in \N }$, $\Vl_n : E\times E \rightarrow \R_+$, a function $r \in \Lambda$
and a constant $b < \infty$ such that \textbf{A}($\Delta, \ell, \Vl_n,r,b$) is
satisfied. Then, for any $x,y \in E$,
\begin{equation}
\label{eq:moment_temps_darret_fini}
\expeMarkovTilde{x,y}{R(T_0) } \leq
c_{1,r} c_{2,r} R(\ell-1) \{ Q^{\ell-1} \Vl_0(x,y) + b r(0) \} \eqsp,
\end{equation}
and $ \sup_{(z,t) \in \Delta}\expeMarkovTilde{z,t}{R(T_0)}$ is finite.
\end{proposition}
\begin{proof}
  By \cite[Proposition~11.3.2]{bible},
  $\expeMarkovTilde{x,y}{\sum_{k=0}^{\tau_\Delta-1} r(k)} \leq \Vl_0(x,y) + b
  r(0)$, where $\tau_\Delta$ is the return time to $\Delta$.  Since $R(k) \leq
  c_{1,r} \sum_{p=0}^{k-1} r(p)$, the previous inequality provides a bound on
  $\expeMarkovTilde{x,y}{R(\tau_\Delta)}$.  The conclusion follows from the
  Markov property upon noting that $R(T_0) \leq c_{2,r} R(\ell-1)
  R(\tau_\Delta \circ \theta^{\ell-1} )  $.
\end{proof}
Combining the strong Markov property, (\ref{eq:def_T_m}) and
\autoref{propo:temps_darret_fini}, it is easily seen that
$\expeMarkovTilde{x,y}{T_m }<\infty$ for any $m \geq 0$ and $x,y \in E$. This
yields the following result.
\begin{corollary}
\label{coro:temps_darret_fini}
Assume that there exist a coupling kernel $Q$ for $P$, $\Delta \in \B(E\times
E)$, $\ell \in \nset^*$, a sequence of measurable functions $\defEns{\Vl_n, n
  \in \N }$, $\Vl_n : E\times E \rightarrow \R_+$, a function $r \in \Lambda$
and a constant $b < \infty$ such that \textbf{A}($\Delta, \ell, \Vl_n,r,b$) is
satisfied. Then, for all $j \geq0$ and $(x,y) \in E \times E$,
$\probaMarkovTilde{x,y}[T_j < \infty] = 1$.
\end{corollary}
For $r \in \Lambda$, there exists $r_0 \in \Lambda_0$ such that
$c_{3,r}= 1 \vee \sup_{t \geq 0} r(t)/r_0(t) < \infty$ and $c_{4,r}= 1 \vee \sup_{t
  \geq 0} r_0(t)/r(t) < \infty$. Denote $c_{5,r} = \sup_{t,u \in
  \R_+} \ r(t+u)/\{r(t)r(u)\}$ and define for $\kappa >0$, the
real $M_\kappa$ such that for all $t \geq M_\kappa$, $r(t) \leq \kappa
R(t)$. $M_\kappa$ is well defined by
\autoref{lem:suite_sous_geometrique}-\eqref{lem:suite_sous_geo3}.
\begin{lemma}
\label{lem:surmartingale_contraction}
Assume that there exist a coupling kernel $Q$ for $P$, $\Delta \in \B(E\times
E)$, $\ell \in \nset^*$, a sequence of measurable functions $\defEns{\Vl_n, n
  \in \N }$, $\Vl_n : E\times E \rightarrow \R_+$, a function $r \in \Lambda$
and constants $\epsilon>0$, $b < \infty$ such that
\autoref{hyp:small-set}($\Delta,\ell,\epsilon$) and \textbf{A}($\Delta, \ell,
\Vl_n,r,b$) are satisfied.  Then,
\begin{enumerate}[(i)]
\item
 \label{item:surmartingale_contraction_1}
 For all $x,y \in E$ and for all $n\in \N$,$m \in \nset^*$,
 \begin{equation*}
 \probaMarkovTilde{x,y}[T_m \geq n] \leq    \{ a_1 Q^{\ell-1}\Vl_0(x,y)+ a_2 \} / R(n/2) +  a_3/R(n/(2m)) \eqsp.
 \end{equation*}
\item  \label{item:surmartingale_contraction_2}
 For all $\kappa >0$,  for  all $x,y \in E$ and for all $n,m \in \N$,
\begin{equation*}
\probaMarkovTilde{x,y}[T_m \geq n]\leq   (1+b_1 \kappa)^m \{ \kappa^{-1} r(M_\kappa) + a_1 Q^{\ell-1}\Vl_0(x,y) + a_2 \} / {R(n)}  \eqsp,
\end{equation*}
\end{enumerate}
The constants $\{a_i\}_{i=1}^3,b_1$ can be directly obtained from the proof.
\end{lemma}
\begin{proof}
  Since $r \in \Lambda$, there exists $r_0 \in \Lambda_0$ such that $c_{3,r} +
  c_{4,r} < \infty$. Denote by $R_0$ the function (\ref{eq:definition:R})
  associated with $r_0$.
\begin{align}
\nonumber
 \probaMarkovTilde{x,y}[T_m \geq n] &\leq   \probaMarkovTilde{x,y}[T_0 \geq n/2] + \probaMarkovTilde{x,y}[T_m- T_0 \geq n/2]\\
\nonumber
& \leq \expeMarkovTilde{x,y}{R(T_0)}/R(n/2) +  \expeMarkovTilde{x,y}{R_0((T_m-T_0)/m)} / R_0(n/(2m)) \\
 \label{eq:bon_poly_log_2}
& \leq \{a_1 Q^{\ell-1} \Vl_0(x,y) + a_2 \}/R(n/2) +  c_{3,r} \expeMarkovTilde{x,y}{R_0((T_m-T_0)/m)} / R(n/(2m)) \eqsp,
 \end{align}
 where we used \autoref{propo:temps_darret_fini} in the last inequality, and
 $a_1 = c_{1,r} c_{2,r} R(\ell-1)$; $a_2 = a_1 b r(0)$. Since $R_0$ is convex
 (see~\autoref{lem:suite_sous_geometrique}), we have by \eqref{eq:def_T_m}:
\begin{equation*}
\expeMarkovTilde{x,y}{R_0((T_m-T_0)/m)}
 \leq  c_{4,r}m^{-1} \expeMarkovTilde{x,y}{ \sum_{k=0}^{m-1} R(T_0 \circ \theta^{T_k})} \eqsp.
\end{equation*}
Using \autoref{coro:temps_darret_fini} and the strong Markov property,  for any $x,y \in E$ and $m \geq 1$,
\begin{equation}
 \label{eq:bon_poly_log}
\expeMarkovTilde{x,y}{R_0((T_m-T_0)/m)}
\leq c_{4,r}  C_\Delta    \eqsp, \qquad \text{with} \ C_\Delta = \sup_{(x,y)
    \in \Delta}\expeMarkovTilde{x,y}{R(T_0)} \eqsp.
\end{equation}
Plugging \eqref{eq:bon_poly_log} in \eqref{eq:bon_poly_log_2} implies
\eqref{item:surmartingale_contraction_1} with $a_3 = c_{3,r} c_{4,r} C_\Delta $.
We now consider \eqref{item:surmartingale_contraction_2}.  Again by the Markov
inequality, since $R$ is increasing,
 \begin{equation}
   \label{eq:bon_sous_expo_2}
 \probaMarkovTilde{x,y}[T_m \geq n] \leq R^{-1}(n)
 \expeMarkovTilde{x,y}{R(T_m)} \eqsp.
 \end{equation}
 If $m =0$, the result follows from \autoref{propo:temps_darret_fini}. If $m \geq 1$,
 using the definitions  of $T_m$ and $R$, given respectively in \eqref{eq:def_T_m} and \eqref{eq:definition:R}, and
since for all $t,u \in \rset_+$, $R(t + u) \leq R(t) + c_{5,r}R(u)r(t)$, we get
 \[
 \expeMarkovTilde{x,y}{R(T_m)} \leq \expeMarkovTilde{x,y}{R(T_{m-1})} +
 c_{5,r} \expeMarkovTilde{x,y}{r(T_{m-1}) R(T_0 \circ \theta^{T_{m-1}})} \eqsp.
 \]
 Thus, by the strong Markov property
 \begin{equation}
   \label{eq:bon_sous_expo_1}
 \expeMarkovTilde{x,y}{R(T_m)} \leq \expeMarkovTilde{x,y}{R(T_{m-1})} + c_{5,r} C_\Delta
 \expeMarkovTilde{x,y}{r(T_{m-1})} \eqsp.
 \end{equation}
 Let $\kappa>0$. Since by definition, for all $t \geq M_\kappa$, $r(t) \leq \kappa R(t)$,
$ \expeMarkovTilde{x,y}{r(T_{m-1})} \leq r(M_\kappa) + \kappa \expeMarkovTilde{x,y}{R(T_{m-1})}$,
so that \eqref{eq:bon_sous_expo_1} becomes
\[
 \expeMarkovTilde{x,y}{R(T_m)} \leq (1+c_{5,r}C_\Delta  \kappa) \expeMarkovTilde{x,y}{R(T_{m-1})}
 + c_{5,r}C_\Delta  r(M_\kappa) \eqsp.
 \]
By a straightforward induction  we get,
\begin{align*}
  \expeMarkovTilde{x,y}{R(T_m)} & \leq  (1+ c_{5,r}C_\Delta \kappa )^m
  (\expeMarkovTilde{x,y}{R(T_0)} + r(M_\kappa)/\kappa) \eqsp.
 \end{align*}
 Plugging this result in \eqref{eq:bon_sous_expo_2} and using
 \autoref{propo:temps_darret_fini} concludes the proof. Note that $b_1= c_{5,r}
 C_\Delta$ and $a_2= c_{1,r}c_{2,r}R(\ell-1)b r(0) $.
\end{proof}
\begin{lemma}
\label{lem:itere_couplage}
Assume that there exist a coupling kernel $Q$ for $P$, $\Delta \in \B(E\times
E)$, $\ell \in \nset^*$, a sequence of measurable functions $\defEns{\Vl_n, n
  \in \N }$, $\Vl_n : E\times E \rightarrow \R_+$, a function $r \in \Lambda$
and constants $\epsilon>0$, $b < \infty$ such that
\autoref{hyp:small-set}($\Delta,\ell,\epsilon$) and \textbf{A}($\Delta, \ell,
\Vl_n,r,b$) are satisfied.  Then,
\begin{enumerate}[(i)]
\item
\label{eq:distance_iteration}
For all $x,y \in E$ and $n \in \N$,
\begin{equation*}
\expeMarkovTilde{x,y}{d(X_n,Y_n)} \leq  1/R(n)+ \{ a_1Q^{\ell-1} \Vl_0(x,y) + a_2 \}/R(n/2)   +  a_3 v^{-1}_n
\eqsp,
\end{equation*}
where $v_n\eqdef R(-n \log(1-\epsilon) /\{2 (\log(R(n))-\log(1-\epsilon))\}) $.
\item \label{eq:distance_iteration_2} For all $\delta \in \ooint{0,1}$, $x,y
  \in E$ and $n \in \N$,
\begin{equation*}
\expeMarkovTilde{x,y}{d(X_n,Y_n)}\leq  \left(1+ (1+b_1 \kappa) \{ \kappa^{-1} r(M_\kappa) + a_1 Q^{\ell-1} \Vl_0(x,y) + b_2 \} \right)/R^\delta(n) \eqsp,
\end{equation*}
where $\kappa = ((1-\epsilon)^{-(1-\delta)/\delta}-1)/b_1$.
\end{enumerate}
The constants $a_i, b_j$ are given by \autoref{lem:surmartingale_contraction}.
\end{lemma}
\begin{proof}
  By \autoref{lem:surmartingale2} and
  \autoref{lem:surmartingale_contraction}-\eqref{item:surmartingale_contraction_1},
  there exist $\{ a_i \}^3_{i=1}$ such that for all $x,y$ in $E$ and for all $n \geq 0$ and
  $m\geq 0$
  \begin{align*}
 &   \expeMarkovTilde{x,y}{d(X_n,Y_n)} \leq (1-\epsilon)^{m} + \probaMarkovTilde{x,y}[T_m \geq n]\\
    & \qquad \leq (1-\epsilon)^{m} +  \{ a_1 Q^{\ell-1}\Vl_0(x,y)+ a_2 \} / R(n/2) +  a_3/R(n/(2m))
 \eqsp.
  \end{align*}
  We get the first inequality by choosing $m = \lceil -
  \log(R(n))/\log(1-\epsilon) \rceil$.  Let us prove
  \eqref{eq:distance_iteration_2}.  Fix $\delta \in \ooint{0,1}$ and choose the
  smallest integer $m$ such that  $ (1-\epsilon)^{m} \leq R(n)^{-\delta}$
  (\ie\ $m= \lceil - \delta \log R(n) / \log(1-\epsilon) \rceil$).  Apply
  \autoref{lem:surmartingale_contraction}-\eqref{item:surmartingale_contraction_2},
  with $\kappa >0$ such that $(1+b_1 \kappa) = (1-\epsilon)^{-
    ((1-\delta)/\delta)}$; hence, upon noting that $R(n)^{-\delta} \leq
  (1-\epsilon)^{m-1}$, it holds
  \begin{multline*}
    (1 + b_1 \kappa)^m = (1 + b_1 \kappa) \left\{(1-\epsilon)^{m-1} \right\}^{-
      ((1-\delta)/\delta)} \leq (1 + b_1 \kappa) \left\{ R(n)^{-\delta}
    \right\}^{- ((1-\delta)/\delta)} = (1 + b_1 \kappa) R(n)^{1-\delta} \eqsp.
  \end{multline*}
\end{proof}

We now prove that \autoref{hyp:drift_double}($\Delta,\phi,V$) implies
\textbf{A}. For a function $\phi \in \mathbb{F}$ and a measurable function $V:E
\to \coint{1,\infty}$, set \begin{equation}
\label{eq:r_phi}
r_\phi(t) = (H_\phi ^{\inv})'(t) = \phi (H_\phi^{\inv}(t)) \eqsp,
\end{equation}
where $H_\phi$ is defined in \eqref{eq:defHphi} and $H_\phi^\inv$ denotes its
inverse; and define for $k \geq 0$, $H_k: [1, \infty) \to \R_+$ and $\Vl_k: E
\times E \to \R_+$ by
\begin{align}
  H_k(u) & = \int_0 ^{H_\phi(u) } r_\phi (t+k) \dint t = H_\phi ^{\inv} (H_\phi(u)+k) - H_\phi ^{\inv} (k) \eqsp, \label{eq:definition:Hk} \\
\Vl_k(x,y) &= H_k\left(V(x)+V(y) \right) \eqsp. \label{eq:definition:Vlk}
\end{align}
Note that $\Vl_k$ is measurable, $H_k$ is twice continuously differentiable on $[1, \infty)$ and  that  $H_0(x) \leq x $ so $\Vl_0(x,y) \leq  V(x)+V(y)$.  
 The proof of the following lemma is adapted from \cite[Proposition~2.1]{douc:fort:moulines:soulier:2004}.
\begin{lemma}
  \label{lem:eqhyp1}
  Assume that there exist $\Delta \in \B(E\times
  E)$, a function $\phi \in \mathbb{F}$ and a measurable function $V: E \to
  \coint{1,\infty}$ such that \autoref{hyp:drift_double}($\Delta,\phi,V$) is
  satisfied.  For any $x,y \in E$ and any coupling $\lambda \in
  \couplage{P(x,\cdot)}{P(y,\cdot)}$ we have:
  \[
  \int_{E \times E} \Vl_{k+1} (z,t) \dint \lambda (z,t) \leq \Vl_k(x,y) -
  r_\phi(k) + \frac{b}{r_\phi(0)} \,  r_\phi(k+1) \1_{\Delta}(x,y) \eqsp,
  \]
  where $r_\phi$ and $\Vl_k$ are defined in \eqref{eq:r_phi} and
  \eqref{eq:definition:Vlk} respectively.
\end{lemma}
\begin{proof}
  Set $\Vl(x,y) = V(x) +V(y)$. By \cite[Proposition
    2.1]{douc:fort:moulines:soulier:2004} $H_{k+1}$ is concave, which implies that for all $u \geq 1$ and $t \in
  \rset$ such that $t + u \geq 1$, we have
  \begin{equation}
    \label{eqhyp1inegHk2}
    H_{k+1}(t+u) - H_{k+1}(u) \leq H'_{k+1}(u) t  \eqsp.
    \end{equation}
    In addition, according to the proof of \cite[Proposition
    2.1]{douc:fort:moulines:soulier:2004}, for every $u \geq 1$ it holds:
 \begin{equation}
   \label{eq:eqhyp1inegHk}
 H_{k+1}(u) - \phi(u)H_{k+1}'(u) \leq H_k(u) - r_\phi(k) \eqsp.
 \end{equation}
 
 Therefore, the Jensen inequality and
 \eqref{eq:drift_noyau_R2} imply
\begin{align*}
  \int_{E\times E} \Vl_{k+1}(z,t) \dint \lambda(z,t) & \leq
  H_{k+1}\left(\int_{E\times E} \Vl(z,t) \dint \lambda(z,t) \right)\\
  & \leq H_{k+1} \left(\Vl(x,y) - \phi \circ \Vl(x,y) + b
    \1_{\Delta}(x,y)\right) \eqsp.
\end{align*}
Using \eqref{eqhyp1inegHk2}, \eqref{eq:eqhyp1inegHk} and the inequality
$H'_{k+1}(\Vl(x,y)) \leq H'_{k+1}(1)$ we get that
\begin{align*}
\int_{E\times E} \Vl_{k+1}(z,t) \dint \lambda(z,t)  &\leq H_{k+1} \parenthese{\Vl(x,y)} - \phi \circ \Vl(x,y) \, H_{k+1}'(\Vl(x,y)) + b
  H_{k+1}'(1) \1_{\Delta}(x,y)\\
& \ \leq H_{k} \parenthese{\Vl(x,y)} - r_\phi(k) + b H_{k+1}'(1)
  \1_{\Delta}(x,y) \eqsp.
\end{align*}
The proof is concluded upon noting that $H_{k+1}'(1) = r_\phi(k+1)/r_\phi(0)$.
\end{proof}
\begin{proposition}
\label{lem:implication_phi_suite_drift}
Assume that there exist a coupling kernel $Q$ for $P$, $\Delta \in \B(E\times
E)$, a function $\phi \in \mathbb{F}$ and a measurable function $V: E \to
\coint{1,\infty}$ such that \autoref{hyp:drift_double}($\Delta,\phi,V$) is
satisfied. Then for any $\ell \geq 0$, \textbf{A}($\Delta,\ell,\Vl_n,r_\phi,
\{\sup_{p \geq 0} r_\phi(p+1)/r_\phi(p) \}\, b /r_\phi(0)$) holds with
$\Vl_n(x,y) = H_n(V(x) + V(y))$ where $r_\phi$ and $H_n$ are given by
\eqref{eq:r_phi} and \eqref{eq:definition:Hk} respectively.
\end{proposition}
\begin{proof}
  By \cite[Lemma 2.3]{douc:fort:moulines:soulier:2004}, $r_\phi \in \Lambda$.  Then, it follows from
  \autoref{lem:eqhyp1} and
  \autoref{lem:suite_sous_geometrique}-\eqref{lem:suite_sous_geo_1} that for all $x,y \in E$,
\[
Q \Vl_{k+1}(x,y)  \leq \Vl_k(x,y) - r_\phi(k)
+ b \left\{ \sup_{p \geq 0}
r_\phi(p+1)/r_\phi(p) \right\}  r_\phi(k) \1_{\Delta}(x,y)/r_\phi(0) \eqsp.
\]
Finally, since $Q^\ell$ is a coupling kernel for $P^\ell$, we have by iterating
the inequality \eqref{eq:drift_noyau_R2}
\[
Q^\ell \Vl_0(x,y) \leq P^\ell V(x) + P^\ell V(y) \leq V(x) + V(y) + \ell b
\eqsp.
\]
Therefore under \autoref{hyp:drift_double}($\Delta,\phi,V$), $\sup_{(x,y) \in
  \Delta} Q^{\ell-1} \Vl_0(x,y) < \plusinfty$.
 \end{proof}

\begin{proof}[Proof of \autoref{theo:convergence_log_double_drift_1}-\eqref{theo:item_rate1}]
  Using \autoref{lem:implication_phi_suite_drift}, \autoref{lem:itere_couplage}
  applies with $R(t) = 1+ \int_{0}^t r_\phi(s) \rmd s$ for $t \in \rset_+$.
  Note that we have $R = H_\phi^\inv$.

  Set $M_V >0 $ such that $ \pi(V \leq M_V) \geq 1/2$; such a constant exists
  since $\pi(E)=1$ and $E = \bigcup_{k \in \N} \defEns{V \leq k}$.  Set $M
  >M_V$ and define the probability $\pi_M$ by $\pi_M(\cdot) = \pi(\cdot \cap
  \{V \leq M \})/ \pi(\{V \leq M \})$. Since $\pi$ is invariant for $P$,
  $W_{d}(P^n(x,\cdot) , \pi ) = W_{d}(P^n(x,\cdot) , \pi P^n )$ and the
  triangle inequality implies:
\begin{equation}
  \label{eq:base_propo_finale_couplage}
  W_{d}(P^n(x,\cdot) , \pi ) \leq W_{d}(P^n(x,\cdot), \pi_M
  P^n) + W_{d} (\pi_M P^n , \pi P^n) \eqsp, \quad \text{for all $n \geq 1$.}
\end{equation}
Consider the first term in the RHS of \eqref{eq:base_propo_finale_couplage}.  By
\autoref{lem:contract_coupling_weak_contraction}-\eqref{lem:contraction_rev_mu_nu_2}, for all $x \in E$ and $n \geq 1$ :
\[
W_{d}(P^n(x,\cdot), \pi_M P^n) \leq \inf_{\lambda \in ~ \couplage{\delta_x}{
    \pi_M} } \int_{E\times E}Q^nd(z,t) \, \dint \lambda
(z,t) \eqsp.
\]
Let $v_n= R(-n \log(1-\epsilon) /\{2 (\log(R(n))- \log(1-\epsilon))\})$.  By
\autoref{lem:itere_couplage}-\eqref{eq:distance_iteration} and since $R=H_\phi
^{\inv}$ is increasing, for all $x \in E$ and $n \geq 1$
\begin{align}
  \nonumber
  & R(n/2) \, W_{d}(P^n(x,\cdot), \pi_M P^n)  \\
  \nonumber & \qquad \leq R(n/2)/R(n) + a_1 \ \inf_{\lambda \in ~
    \couplage{\delta_x}{ \pi_M} } \int_{E\times E} ( P^{\ell-1}V(z) +
  P^{\ell-1 }V(t) ) \, \dint \lambda (z,t) + a_2  +   a_3 R(n/2)/ v_n  \\
\label{eq:deuxieme_terme_majo_1_theo_poly}
& \qquad \leq a_1 \left(V(x) + \int_E V(t) \rmd \pi_M(t) + b (\ell-1) \right)+
a_2 +1 + a_3 R(n/2) / v_n \eqsp,
\end{align}
where  in the last inequality, we used
\begin{equation}
  \label{eq:driftitere}
  P^k V(x) \leq V(x) + b k /2 \eqsp.
\end{equation}
which is obtained by iterating the drift inequality~(\ref{eq:drift_noyau_R2})
and applying it with $x=y$. Since $ x \mapsto \phi(x) /x$ is non-increasing,
$V(t) \leq M \phi(V(t))/\phi(M)$ on $\defEns{V \leq M}$, we have
\begin{equation}
\label{eq:majo_int_V}
\int_E V(t) \rmd \pi_M(t) \leq 2 \pi(\phi \circ V)  M / \phi(M) \eqsp.
\end{equation}
Note that by \autoref{coro:existence_pi}, $M_\phi = \int_E \phi \circ V(t) \ \rmd \pi(t) <
\infty$.  Combining \eqref{eq:deuxieme_terme_majo_1_theo_poly} and
\eqref{eq:majo_int_V} yield
\begin{equation}
  \label{eq:deuxieme_terme_ineg_propo_finale_couplage}
W_{d}(P^n(x,\cdot), \pi_M P^n)
\leq  \{ a_1 \left(V(x) + 2  M_\phi  M /\phi(M)  + b (\ell-1)
\right)+ a_2 +1 \}/R(n/2) +   a_3/ v_n \eqsp.
\end{equation}
Consider the second term in the RHS of \eqref{eq:base_propo_finale_couplage}.
Since $d$ is bounded by $1$, $W_d(\mu,\nu) \leq W_{d_0}(\mu,\nu)$ (where
$W_{d_0}$ is the total variation distance) and
\autoref{lem:contract_coupling_weak_contraction}-\eqref{lem:contract_rev_3}
implies $ W_{d} ( \pi_M P^n , \pi P^n) \leq W_{d}(\pi_M, \pi) \leq W_{d_0}(
\pi_M , \pi )$. For every $A \in \B(E)$, we get
\[
\left| \pi_M(A) - \pi(A) \right| = \left| \pi_M(A) (1 - \pi(\{V \leq M\})) + \pi_M(A) \pi(V \leq M) -\pi(A) \right|  \leq 2 \pi(\{V > M \}) \eqsp,
\]
showing that
\begin{equation}
\label{eq:alain-1}
W_{d} ( \pi_M P^n , \pi P^n) \leq 2 \pi(\{V>M \}) = 2\pi\parenthese{\{\phi(V) >\phi(M) \}} \leq  2 M_\phi/\phi(M) \eqsp.
\end{equation}
Since $R(n/2) > M_V$ for all $n$ large enough, we can now choose $M = R(n/2)$
in \eqref{eq:deuxieme_terme_ineg_propo_finale_couplage} and \eqref{eq:alain-1}.
This yields
\begin{equation*}
W_d(P^n(x,\cdot), \pi) \leq \{ a_1 \left(V(x)  + b (\ell-1)
\right)+ a_2 +1 \}/H_\phi ^{\inv}(n/2) + 2 M_\phi(a_1+1)/\phi(R(n/2))+  a_3/ v_n \eqsp.
\end{equation*}

\eqref{theo:item_rate2} The proof is along the same lines, using
\autoref{lem:itere_couplage}-\eqref{eq:distance_iteration_2} instead of
\autoref{lem:itere_couplage}-\eqref{eq:distance_iteration}. Finally, we end up
with the following inequality for $n$ large enough:
\begin{multline*}
  W_d(P^n(x,\cdot), \pi) \leq (1+(1+b_1 \kappa ) \{ \kappa^{-1} r_\phi (M_\kappa) + a_1(V(x) +b(\ell-1)) +b_2 \} )/ \{ R^\delta(n) \} \\
  + 2M_\phi((1+b_1 \kappa)a_1+1)/\{ \phi( R^\delta(n)) \} \eqsp,
\end{multline*}
where $\kappa = ((1-\epsilon)^{-(1-\delta)/\delta}-1)/b_1 $.
\end{proof}

\subsection{Proof of \autoref{prop:H1toDriftDouble}}
\label{subsec:proof:prop:H1toDriftDouble}
Note that since $c= 1 - 2 b / \phi(\upsilon)$ and $\upsilon > \phi^{\inv}(2b)$, we get $c \in \ooint{0,1}$.
Set $\smallSet = \{ V \leq \upsilon \}$. By \eqref{eq:drift},
  \[
  PV(x) + PV(y) \leq V(x) + V(y) - c \phi\parenthese{V(x) + V(y) } + 2b \1_{ \smallSet
    \times \smallSet } (x,y)+ \Omega(x,y)
  \]
  where $\Omega(x,y) = c \phi\parenthese{V(x) + V(y) } - \phi(V(x)) -\phi(V(y))
  + 2b \1_{\parenthese{\smallSet \times \smallSet }^c}(x,y)$.  We show that for
  every $x,y \in E$, $\Omega(x,y) \leq 0$.  Since $\phi$ is sub-additive (note
  that $\phi(0) = 0$), for all $x,y \in E$
\[
  \Omega(x,y)  \leq -(1-c)\parenthese{ \phi(V(x)) +\phi(V(y)) } + 2b  \1_{\parenthese{\smallSet \times \smallSet }^c}(x,y) \eqsp.
\]
On $\parenthese{\smallSet \times \smallSet}^c$, $\phi(V(x)) +
\phi(V(y)) \geq \phi(\upsilon)$. The definition of $c$ implies that
$\Omega(x,y) \leq 0$.


%% file: ProofApplication.tex
\section{Proofs of \autoref{sec:PCN}}
\label{sec:proof:application}
\label{appendix:GaussianMeasure}

\begin{lemma}
\label{lem:prolongement_phi}
Let $M>0$. Assume that there exists an increasing continuoulsy differentiable
concave function $\phi: [M,\infty) \to \rset_+$, such that $\lim_{x \to \infty}
\phi'(x)= 0$ and satisfying, on $\{ V \geq M\}$, $PV(x) \leq V(x) - \phi \circ
V(x) + b$. Then, there exist $\tilde{\phi} \in \mathbb{F}$ and $\tilde{b}$ such
that, $PV \leq V - \tilde{\phi} \circ V + \tilde{b}$ on $E$, $\phi(v) =
\tilde{\phi}(v)$ for all $v$ large enough, and $\tilde{\phi}(0) = 0$.
\end{lemma}
\begin{proof}
Observe indeed that the function $\tilde{\phi}$ defined by
\[
\tilde{\phi}(t) =
\begin{cases}
(2 \phi'(M) - \frac{\phi(M)}{M}) t +
\frac{2(\phi(M) - M \phi'(M))}{\sqrt{M}} \sqrt{t}
& \text{ for } 0 \leq t < M \\
\phi(t) & \text{ for } t \geq M  \eqsp,
\end{cases}
\]
is concave increasing and continuously differentiable on
$\coint{1,\plusinfty}$, $\tilde{\phi}(0)=0$, $\lim_{v \to \infty}
\tilde{\phi}(v) = \infty$ and $\lim_{v \to \infty} \tilde{\phi}'(v)
=0$.  The drift inequality (\ref{eq:drift_noyau_R2}) implies that for all $x
\in E$
\begin{equation*}
PV(x) \leq V(x)  - \tilde{\phi}\parenthese{ V(x)} +
\tilde{b} \eqsp,
\end{equation*}
with $\tilde{b}= b + \sup_{\{t  \leq M\}}\defEns{
  \tilde{\phi}\parenthese{t} - \phi \parenthese{t}}$.
\end{proof}

\subsection{Proof of \autoref{lem:drift_g_hold}}
\label{sec:proof:lem:drift_g_hold}
For notational simplicity, let $P= P_{\pCN}$.  By definition of $P$, $V(X_1)
\leq V(X_0) \vee V(\rho X_0 +\sqrt{1-\rho^2} Z_1)$. Since $\norm{x+y}^2 \leq 2
\norm{x}^2 + 2 \norm{y}^2$, we get
 \begin{equation}
 \label{eq:majo_sur_K_g_hold}
 \sup_{x \in \boule{0}{1}} PV(x) \leq \sup_{x \in \boule{0}{1}}
 \int_\hilbert \exp\parenthese{2s\parenthese{ \norm{x}^2 +( 1-\rho^2) \norm{z}^2}} \dint \gamma(z) \eqsp,
 \end{equation}
 and \autoref{theo:Fernique} implies that the RHS is finite.\\
Now, let $x \not \in \boule{0}{1}$ and  set $w(x) = (1-\rho) \norm{x}/2$. Define the events $\setAbruit= \{ \norm{Z_1} \leq w(X_0)/\sqrt{1-\rho^2} \}$,
$\setaccept = \{\alpha(X_0,\rho X_0 + \sqrt{1-\rho^2}Z_1) \geq U\}$, and
$ \setreject = \{\alpha(X_0, \rho X_0 + \sqrt{1-\rho^2}Z_1) < U\}$,
where $U \sim \mathcal{U}(\ccint{0,1})$, $Z_1 \sim
\gamma$, and $U$ and $Z_1$ are independent. With these definitions, we get,
\begin{equation}
\label{eq:drift_g_hold}
PV(x) = \expeMarkov{x}{V(X_1) \1_{ \setAbruit^c}} + \expeMarkov{x}{V(X_1)\1_{\setAbruit} ( \1_{\setaccept} + \1_{\setreject}) } \eqsp.
\end{equation}
For the first term in the RHS, using again $V(X_1) \leq V(X_0) \vee V(\rho X_0 +\sqrt{1-\rho^2} Z_1)$ and $\norm{x+y}^2 \leq 2 \norm{x}^2 + 2 \norm{y}^2$, we get
\begin{align}
\nonumber
\expeMarkov{x}{V(X_1) \1_{ \setAbruit^c}}
&\leq \exp \parenthese{2 s  \norm{x}^2 }
\int_{\sqrt{1-\rho^2} \norm{z}  \geq w(x) } \exp \parenthese{
2s (1-\rho^2) \norm{z}^2 } \dint \gamma (z)\\
\nonumber
& \leq \exp \parenthese{2 s  \norm{x}^2 -(\theta/2) w(x)^2}
\int_{\hilbert } \exp \parenthese{
(\theta/2 + 2s)( 1-\rho^2) \norm{z}^2 } \dint \gamma (z) \\
\nonumber
& \leq \int_{\hilbert } \exp
( (5/8) (1-\rho^2)\theta \norm{z}^2 ) \dint \gamma (z) \eqsp,
\end{align}
where the definition of $s$ and $w$ are used for the last inequality.
Hence by \autoref{theo:Fernique}, there exists a constant $b< \infty$ such that
\begin{equation}\label{eq:drift_g_hold_2}
  \sup_{x \in \hilbert} \expeMarkov{x}{V(X_1) \1_{\setAbruit^c}}  \leq b \eqsp.
\end{equation}
Consider the second term in the RHS of \eqref{eq:drift_g_hold}. On
the event $\setaccept  \cap \setAbruit$, the move is accepted and $\norm{X_1 - \rho X_0 } \leq w(X_0)$.
On $\setreject$, the move is rejected and  $X_1 = X_0$. Hence,
\begin{equation*}
\expeMarkov{x}{V(X_1)\1_{\setAbruit} ( \1_{\setaccept} + \1_{\setreject})
}
\leq \left\{\sup_{z \in \boule{\rho x}{w(x)}} V(z) \right\} \probaMarkov{x}{\setAbruit \cap \setaccept } +
V(x) \probaMarkov{x}{\setAbruit \cap \setreject}
\eqsp.
\end{equation*}
For $z \in \boule{\rho x}{w(x)}$, by the triangle inequality, $ V(z) \leq \exp(s(1+\rho)^2 \norm{x}^2/4)$. Therefore
for any $x \not \in \boule{0}{1}$ since $\rho \in \coint{0,1}$, $\sup_{z \in \boule{ \rho x}{w(x)}} V(z) \leq \zeta V(x) $, with $\zeta = \exp\{((1+\rho)^2/4-1) s \} <1$.  This yields
\begin{align*}
  \expeMarkov{x}{V(X_1)\1_{\setAbruit} ( \1_{\setaccept} + \1_{\setreject}) } & \leq \zeta V(x) \probaMarkov{x}{ \setAbruit \cap \setaccept } + V(x) \probaMarkov{x}{ \setAbruit \cap \setreject }\\
  & \leq V(x) \probaMarkov{x}{ \setAbruit} -(1-\zeta) V(x) \probaMarkov{x}{\setaccept \cap \setAbruit} \eqsp.
 \end{align*}
Since $U_1$ and $Z_1$ are independent, we get
$$
\probaMarkov{x}{\setaccept \cap \setAbruit} = \expeMarkov{x}{\parenthese{ 1 \wedge \rme^{g(x) - g(\rho x +
      \sqrt{1-\rho^2} Z_1)}} \1_{\setAbruit} } \eqsp.
$$
By definition of the set
$\setAbruit$ and using the inequality
$\inf_{z \in \boulefermee{\rho x }{ w(x)}} \exp(g(x) - g(z)) \geq \exp(-  C_g  (1-\rho)^\beta (3/2)^{\beta} \norm{x}^\beta)$, we get
$ \probaMarkov{x}{\setaccept \cap \setAbruit} \geq  \exp(- \{ \ln V(x) / \kappa\}^{\beta/2}) \probaMarkov{x}{ \setAbruit}$, with $\kappa =  \theta C_g^{-2/\beta}/36$.
Hence, for any $x \notin \boule{0}{1}$,
\begin{align} \label{eq:drift_g_hold_1}
  \expeMarkov{x}{V(X_1)\1_{\setAbruit} ( \1_{\setaccept} +
    \1_{\setreject}) } \leq  V(x) - (1-\zeta) \,  V(x)  \, \exp(-\kappa^{-\beta/2}\log^{\beta/2} V(x)) \eqsp.
\end{align}
Combining \eqref{eq:majo_sur_K_g_hold}, \eqref{eq:drift_g_hold_2} and \eqref{eq:drift_g_hold_1} in
\eqref{eq:drift_g_hold}, it follows that  there exists $\tilde b >0$ such that, for every $x \in \hilbert$,
\[
PV(x) \leq V(x) - (1-\zeta) \, V(x) \, \exp(-\kappa^{-\beta/2}  \log^{\beta/2} V(x)) + \tilde b \eqsp.
\]
The proof follows from  \autoref{lem:prolongement_phi}.

\subsection{Proof of \autoref{lem:g_simple_smallness_hold}}
\label{proof:lem:g_simple_smallness_hold}
 We preface the proof of \autoref{lem:g_simple_smallness_hold} by a Lemma.
\begin{lemma}
\label{lem:pCN_contrafaible_g_hold}
Assume \autoref{hyp:pCN}. There exists $\eta \in \ooint{0,1}$ satisfying the following assertions
\begin{enumerate}[(i)]
\item \label{eq:contraction_strict_g_hold}
For all $L > 0$, there exists $k(Q_{\pCN},L,\eta) < 1$ such that, for all $x,y \in \boule{0}{L}$ satisfying $d_\eta(x,y) < 1$, $Q_{\pCN}d_{\eta}(x,y)\leq  k(Q_{\pCN},L,\eta)  d_\eta(x,y)$.
\item \label{eq:contraction_large_g_hold}
For all $x,y \in \hilbert$, $Q_{\pCN}d_{\eta}(x,y)\leq \, d_\eta(x,y)$.
\end{enumerate}
\end{lemma}
\begin{proof}
\label{proof:lem:pCN_contrafaible_g_hold}
Let $\eta \in \ooint{0,1}$; for ease of notation, we simply write  $Q$ for $Q_{\pCN}$.  Let $L>0$ and choose $x,y \in \boule{0}{L}$ satisfying $d_\eta(x,y) < 1$.  Let
$(X_1,Y_1)$ be the basic coupling between $P(x,\cdot)$ and $P(y,\cdot)$; let
$Z_1, U_1$ be the Gaussian variable and the uniform variable used for the
basic coupling. Set $\setAbruit = \defEns{\sqrt{1-\rho^2} \norm{Z_1 } \leq 1}$,
$\setaccept  = \defEns{\Psi_\wedge(X_0,Y_0,Z_1)  > U_1}$,  $\setreject = \defEns{\Psi_\vee(X_0,Y_0,Z_1) < U_1}$,
where
\begin{align}
\label{eq:definition-Psi-wedge}
\Psi_\wedge(x,y,z) &= \alpha(x, \rho x + \sqrt{1-\rho^2}z) \wedge \alpha(y, \rho y + \sqrt{1-\rho^2}z) \eqsp \\
\label{eq:definition-Psi-vee}
\Psi_\vee(x,y,z) &= \alpha(x, \rho x + \sqrt{1-\rho^2}z) \vee \alpha(y, \rho y + \sqrt{1-\rho^2}z) \eqsp.
\end{align}
On the event $\setaccept$, the moves are both accepted so that $X_1 =
\rho X_0 + \sqrt{1-\rho^2} Z_1$ and $Y_1 = \rho X_0 + \sqrt{1-\rho^2} Z_1$; On the event
$\setreject$, the moves are both rejected so that $X_1=X_0$ and $Y_1 =Y_0$. It holds,
\begin{equation}
 Q d_\eta(x,y) \leq  \expeMarkovTilde{x,y} {d_\eta(X_1,Y_1)}
 \leq \expeMarkovTilde{x,y} {d_\eta(X_1,Y_1) \1_{\setaccept \cup \setreject}}
 + \probaMarkovTilde{x,y}[(\setaccept \cup \setreject)^c] \eqsp,
\label{eq:res_inter_pCN_g_hold_a}
\end{equation}
where we have  used $d_\eta$ is bounded by $1$.
Since  $d_\eta(X_1,Y_1) = \rho^\beta d_\eta(X_0,Y_0)$, on $\setaccept$, and
$d_\eta(X_1,Y_1) = d_\eta(X_0,Y_0)$, on $\setreject$, we get
$\expeMarkovTilde{x,y} {d_\eta(X_1,Y_1) (\1_{\setaccept \cup \setreject})}
\leq \rho^\beta d_\eta(x,y)\probaMarkovTilde{x,y}[\setaccept] + d_\eta(x,y)\probaMarkovTilde{x,y}[\setreject]$.
Since $\probaMarkovTilde{x,y}[\setaccept] +\probaMarkovTilde{x,y}[\setreject] \leq 1$, we have
\begin{align}
\nonumber
     \expeMarkovTilde{x,y} {d_\eta(X_1,Y_1) (\1_{\setaccept \cup \setreject})}
    &\leq d_\eta(x,y) - (1-\rho^\beta) \, d_\eta(x,y) \probaMarkovTilde{x,y}[\setaccept] \\
    \label{eq:majo1}
    &\leq d_\eta(x,y) - (1-\rho^\beta) \, d_\eta(x,y)
    \probaMarkovTilde{x,y}[\setaccept \cap \setAbruit] \eqsp.
  \end{align}
Set $  \Theta(x,y,z) = \left|\alpha(x,\rho x + \sqrt{1-\rho^2} z) - \alpha(y,\rho y + \sqrt{1-\rho^2} z) \right|$.
Since  $Z_1$ and $U_1$ are independent, it follows that
$\probaMarkovTilde{x,y}[ (\setaccept \cup \setreject)^c ]  \leq \int_\hilbert \Theta(x,y,z) \dint \gamma(z) $
Plugging this identity and \eqref{eq:majo1}  in
\eqref{eq:res_inter_pCN_g_hold_a} yields
\begin{equation}
\label{eq:res_inter_pCN_g_hold_b}
Qd_\eta(x,y) \leq d_\eta(x,y)
- (1-\rho^\beta) d_\eta(x,y) \probaMarkovTildeDeux{x,y}
  {\setaccept \cap \setAbruit}
+ \int_{\hilbert} \Theta(x,y,z) \dint \gamma (z) \eqsp.
\end{equation}
Let us now define $h : \hilbert \to \R$ by
\begin{equation}
\label{eq:def_h}
h(z) = g(z) -g(\rho z) \eqsp.
\end{equation}
We bound from below $\probaMarkovTilde{x,y}[\setaccept \cap \setAbruit]$. Since $U_1$ is independent of $Z_1$, it follows
that
\[
\probaMarkovTilde{x,y}[\setaccept \cap \setAbruit]  \geq \expeMarkovTilde{x,y}{\Psi_\wedge(X_0,Y_0,Z_1)  \1_{\setAbruit}} \eqsp.
\]
By \autoref{hyp:pCN}, for all $z \in \hilbert$ such that $\sqrt{1-\rho^2}\norm{z} \leq 1$, it holds
for $\varpi \in \hilbert$, $g(\varpi) - g(\rho \varpi + \sqrt{1-\rho^2} z) \geq h(\varpi) - C_g$.
Then,
\begin{equation*}
\Psi_\wedge(x,y,z)  \geq 1\wedge( \rme^{-C_g} \rme^{h(x)}) \wedge (\rme^{-C_g} \rme^{h(y)} )
\geq \rme^{-C_g} \parentheseDeux{1 \wedge \rme^{h(x) \wedge h(y)}} \eqsp.
\end{equation*}
Therefore,
\begin{equation}
\label{eq:res_inter_pCN_g_hold_e}
\probaMarkovTilde{x,y}[\setaccept \cap \setAbruit] \geq
\rme^{-C_g} \parentheseDeux{1 \wedge \rme^{h(x)\wedge h(y)}}\probaMarkovTilde{x,y}[\setAbruit] \eqsp.
\end{equation}
We now upper bound the integral term in \eqref{eq:res_inter_pCN_g_hold_b}.
For $x,y \in \hilbert$, define the  partition of $\hilbert$,
\begin{align*}
\Kfrac_1(x,y) &= \{z \in \hilbert: \alpha(x, \rho x + \sqrt{1-\rho^2} z) =
\alpha(y, \rho y + \sqrt{1-\rho^2} z) =1\} \\
 \Kfrac_2(x,y) &= \{z \in \hilbert: \alpha(x, \rho x + \sqrt{1-\rho^2} z) =
1 > \alpha(y, \rho y + \sqrt{1-\rho^2} z)\} \\
 \Kfrac_3(x,y) &= \{z \in \hilbert: \alpha(y, \rho y + \sqrt{1-\rho^2} z) =
1 > \alpha(x, \rho x + \sqrt{1-\rho^2} z)\} \\
\Kfrac_4(x,y) &= \{z \in \hilbert: \alpha(y, \rho y + \sqrt{1-\rho^2} z) <
1 \text{ and }  \alpha(x, \rho x + \sqrt{1-\rho^2} z) < 1\} \eqsp.
\end{align*}
Since on  $\Kfrac_1(x,y)$, $\Theta(x,y,z) = 0$,
\begin{equation}
\label{eq:debut_partition}
\int_\hilbert \Theta(x,y,z) \dint \gamma ( z)=
\sum_{j=2} ^4 \int_{\Kfrac_j(x,y)} \Theta(x,y,z) \dint \gamma ( z) \eqsp.
\end{equation}
For any $a,b >0$, we have
$\abs{a-b} =  (a \vee b ) \parentheseDeux{1 - \parenthese{(a/b) \wedge (b/a)}}$.
Upon noting that $1 - \rme^{-t} \leq t $ for any $t \geq 0$, we have
\begin{equation*}
\Theta(x,y,z) \leq \Psi_{\vee}(x,y,z) \left\vert g(y) - g(x) -g( \rho y + \sqrt{1-\rho^2} z) +g( \rho x + \sqrt{1-\rho^2} z)  \right\vert \1_{\cup_{i=2}^4 \Kfrac_i(x,y)} ( z )  \eqsp.
\end{equation*}
By \autoref{hyp:pCN}, this yields, for $x,y \in \hilbert$ such that $d_\eta(x,y) < 1$,
\begin{equation}
\label{eq:bound-Delta}
\Theta(x,y,z) \leq 2 C_g \norm{y-x}^\beta \Psi_{\vee}(x,y,z) \leq  2  C_g \eta d_\eta(x,y) \Psi_{\vee}(x,y,z) \eqsp.
\end{equation}
On $\Kfrac_2(x,y)$, $g(x) > g(\rho x + \sqrt{1-\rho^2} z)$ and, together with the definition~\eqref{eq:def_h},  this implies that $h(x) \geq g(\rho x + \sqrt{1-\rho^2} z) - g(\rho x )$. Therefore, since under \autoref{hyp:pCN}, $h(x)
\geq -C_g (1-\rho^2)^{\beta/2} \norm{z}^\beta$ we get
\begin{multline}
\label{eq:partition_1}
\int_{\Kfrac_2(x,y)} \Theta(x,y,z) \dint \gamma (z ) \leq 2 C_g \eta d_\eta(x,y)
\int_{\Kfrac_2(x,y)} \dint \gamma ( z ) \\
\leq 2 C_g \eta d_\eta(x,y) \defEns{\parentheseDeux{\rme^{h(x) } \int_{\Kfrac_2(x,y)} \rme^{C_g
    (1-\rho^2)^{\beta/2}\norm{z}^\beta} \dint \gamma ( z )} \wedge 1 }
\leq C_I \eta d_\eta(x,y) \defEns{\rme^{h(x)} \wedge 1} \eqsp,
\end{multline}
for a constant $C_I$, which is finite
according to \autoref{theo:Fernique}.
By symmetry, on $\Kfrac_3(x,y)$,
\begin{equation}
  \label{eq:partition_2}
  \int_{\Kfrac_3(x,y)} \Theta(x,y,z) \dint \gamma (z ) \leq C_I \eta d_\eta(x,y)
  \defEns{\rme^{h(y)} \wedge 1} \eqsp.
  \end{equation}
On $\Kfrac_4(x,y)$, using \autoref{hyp:pCN},
\[
\alpha(x, \rho x + \sqrt{1-\rho^2 } z ) = \rme^{g(x) - g(\rho x + \sqrt{1-\rho^2} z )} \wedge 1
\leq  \parenthese{\rme^{h(x)} \rme^{C_g (1-\rho^2)^{\beta/2}\norm{z}^\beta}} \wedge 1 \eqsp;
\]
and by symmetry, we obtain a similar upper bound for $\alpha(y, \rho y +\sqrt{1-\rho^2}z)$.
Since $\rme^{C_g (1-\rho^2)^{\beta/2}\norm{z}^\beta} \geq 1$, these two inequalities imply
$\Psi_\vee(x,y,z ) \leq  \rme^{C_g (1-\rho^2)^{\beta/2}\norm{z}^\beta}(\rme^{h(x) \vee h(y) }\wedge 1)$.
Hence, using again \eqref{eq:bound-Delta} and \autoref{theo:Fernique}, there exists $C_I < \plusinfty$ such that
\begin{equation}
\label{eq:partition_3}
\int_{\Kfrac_4(x,y)} \Theta(x,y,z) \dint \gamma(z) \leq C_I \eta d_\eta(x,y) \parentheseDeux{\rme^{h(x) \vee
      h(y) } \wedge 1} \eqsp.
\end{equation}
Plugging \eqref{eq:partition_1}, \eqref{eq:partition_2}, \eqref{eq:partition_3} into
\eqref{eq:debut_partition}, we finally obtain
\[
    \int_\hilbert \Theta (x,y,z) \dint \gamma (z) \leq 3 C_I \eta d_\eta(x,y) \parentheseDeux{\rme^{h(x) \vee
        h(y) } \wedge 1} \eqsp.
\]
Finally, under \autoref{hyp:pCN}, for every $x,y \in \hilbert$ such that $d_\eta(x,y) <1$, $\abs{h(x) - h(y) } \leq 2 C_g \norm{x-y}^\beta \leq 2 C_g \eta^\beta$.
Therefore $\rme^{h(x) \vee h(y) } \wedge 1 \leq  \rme^{2 C_g \eta^\beta} \parentheseDeux{\rme^{h(x) \wedge h(y)} \wedge 1}$
and
\begin{equation}
\label{eq:fin_partition}
\int_\hilbert \Theta (x,y,z) \dint \gamma(z)  \leq   3 C_I \rme^{2 C_g \eta^\beta} \, \eta d_\eta(x,y) \parentheseDeux{\rme^{h(x) \wedge h(y) } \wedge 1} \eqsp.
\end{equation}
Plugging \eqref{eq:res_inter_pCN_g_hold_e} and \eqref{eq:fin_partition} in
(\ref{eq:res_inter_pCN_g_hold_b}) yields
\begin{equation*}
Qd_\eta(x,y) \leq d_\eta(x,y) \left(1 - \left\{ (1-\rho^\beta)\rme^{-C_g}\probaMarkovTilde{x,y}[\setAbruit]-  3 C_I \rme^{2 C_g \eta^\beta} \eta
\right\} \parentheseDeux{\rme^{h(x) \wedge h(y)} \wedge 1} \right) \eqsp.
\end{equation*}
Note that $M=\probaMarkovTilde{x,y}[\setAbruit]$ is a positive quantity that does not depend on $x,y$.
Therefore, we may choose $\eta$ sufficiently small so that, for every $x,y \in \hilbert$ satisfying $d_\eta(x,y) < 1$,
\begin{equation}
\label{eq:res_inter_pCN_g_hold_g}
Qd_\eta(x,y) \leq d_\eta(x,y) \left(1 - (1/2)(1-\rho^\beta)\rme^{-C_g}M\parentheseDeux{\rme^{h(x) \wedge h(y)} \wedge 1} \right) \eqsp,
\end{equation}
which implies \autoref{lem:pCN_contrafaible_g_hold}-\eqref{eq:contraction_strict_g_hold} upon noting that, under the stated assumptions,  $\inf_{\boule{0}{L}} h > -\infty$.

We now consider  \eqref{eq:contraction_large_g_hold}. For every $x,y \in \hilbert$, $d_\eta(x,y) \leq 1$, which implies that $Q d_\eta(x,y) \leq 1$. For every $x,y \in \hilbert$ such that $d_\eta(x,y)=1$, $Qd_\eta(x,y) \leq 1=d_\eta(x,y)$. If $d_\eta(x,y) < 1$, \eqref{eq:res_inter_pCN_g_hold_g} shows that
$Qd_\eta(x,y) \leq d_\eta(x,y)$.
\end{proof}

\begin{proof}[Proof of \autoref{lem:g_simple_smallness_hold}]
Let $\sequencenDouble{X}{Y}$ be a Markov chain with Markov kernel $Q$ given by
\eqref{eq:def_Q_pCN}. We denote for all $n \in \N^*$, $Z_n$ and $U_n$, respectively the
common Gaussian variable and uniform variable, used in the definition $(X_n,Y_n)$. Note that by
definition the variables $\defEns{Z_n , U_n ; \eqsp  n \in \N }$
are independent.

Since $\defEns{x: V(x) \leq u } = \defEns{x: \norm{x} \leq (s\log(u))^{1/2}}$, for $u \geq 1$, we only prove that for
all $L >0$, there exist $\ell \in \N^*$ and $\epsilon > 0$ such that
$\boulefermee{0}{L}^2$ is a $(\ell, \epsilon , d_\eta)$-coupling set.
By \autoref{lem:pCN_contrafaible_g_hold}-\eqref{eq:contraction_strict_g_hold}, for any  $L >0 $, there exists $k(Q,L,\eta) \in (0,1)$
such that for any  $x,y \in \boulefermee{0}{L}$  satisfying  $d_\eta(x,y) <
1$, $Qd_\eta(x,y) \leq k(Q,L,\eta) d_\eta(x,y)$.  Then by \autoref{lem:pCN_contrafaible_g_hold}-\eqref{eq:contraction_large_g_hold} , for every $n \in \N^*$,
\begin{equation}
  \label{eq:pCN_final_2}
Q^nd_\eta(x,y) \leq Q^{n-1}d_\eta(x,y) \leq \cdots \leq k(Q,L,\eta) d_\eta(x,y) \eqsp.
\end{equation}
Consider now the case $d_\eta(x,y) = 1$.
Let $n \in \N^*$ and denote for
all $1 \leq i \leq n$
$\setaccept_i = \defEns{U_i \leq \Psi_\wedge(X_{i-1},Y_{i-1},Z_i)}$ and $\setacceptdeux_i(n) = \bigcap_{1 \leq j \leq i}
\parenthese{\{\sqrt{1-\rho^2}\norm{Z_j} \leq L/ n\} \cap \setaccept_j}$ where $\Psi_\wedge$ is defined in \eqref{eq:definition-Psi-wedge}

On the event  $\setacceptdeux_i(n)$, $X_j = \rho X_{j-1} + \sqrt{1-\rho^2} Z_j$ and
$Y_j = \rho Y_{j-1} + \sqrt{1-\rho^2} Z_j$ for all $1 \leq j \leq i$. Then,
since $d_\eta(x,y) \leq \eta^{-1} \norm{x-y}^\beta$, on
$\setacceptdeux_n(n)$ it holds $d_\eta(X_n,Y_n) \leq \eta^{-1} \rho^{\beta n }
\norm{X_0-Y_0}^\beta$. This inequality and $d_\eta(x,y) \leq 1$ yield 
\begin{align}
\nonumber
Q^nd_\eta(x,y)
&= \expeMarkovTilde{x,y}{d_\eta(X_n,Y_n)( \1_{\setacceptdeux_n(n)} + \1_{(\setacceptdeux_n(n))^c})} \leq \eta^{-1} \rho^{\beta n } \norm{x-y}^\beta  \probaMarkovTilde{x,y}[\setacceptdeux_n(n)] + \probaMarkovTilde{x,y}[(\setacceptdeux_n(n))^c] \\
\label{eq:final_pCN}
& \leq \eta^{-1} \rho^{\beta n } (2L)^\beta \probaMarkovTilde{x,y}[\setacceptdeux_n(n)] +
\probaMarkovTilde{x,y}[(\setacceptdeux_n(n))^c]   \leq  1+ \left(\eta^{-1}\rho^{\beta n} (2 L)^\beta -1 \right) \, \probaMarkovTilde{x,y}[\setacceptdeux_n(n)]  \eqsp.
\end{align}
As $\rho \in \coint{0,1}$, there exists $\ell$ such that, $\eta^{-1} \rho^{\beta \ell }
(2L)^\beta < 1$. It remains to lower bound
$\probaMarkovTildeDeux{x,y}{\setacceptdeux_\ell(\ell)}$ by a positive constant to
conclude. Since the random variables $\defEns{ (Z_i , U_i) ; \eqsp i \in \N^*}$ are independent, we get
\begin{multline*}
\probaMarkovTilde{x,y}[\setacceptdeux_\ell(\ell)] =
\probaMarkovTilde{x,y}[
\setacceptdeux_{\ell-1}(\ell) \cap \{\sqrt{1-\rho^2}\norm{Z_{\ell}} \leq L/\ell\}
]\\
\times
\expeMarkovTilde{x,y}{\Psi_\wedge(X_{\ell-1}, Y_{\ell-1}, Z_\ell)
\sachant \setacceptdeux_{\ell-1}(\ell) \cap \{\sqrt{1-\rho^2}\norm{Z_{\ell}} \leq L/\ell\}} \eqsp.
\end{multline*}
For all $1 \leq i \leq \ell$, on the event $ \bigcap_{j \leq i} \left\{\sqrt{1-\rho^2}\norm{Z_{j}} \leq L/\ell \right\}$, it holds
\begin{equation*}
\Psi_\wedge(X_{i-1},Y_{i-1},Z_i) \geq \exp \left( -\sup_{z \in \boule{0}{2L}} g(z) + \inf_{z \in \boule{0}{2L}} g(z) \right) = \delta \eqsp,
\end{equation*}
where $\delta \in (0,1)$. Therefore, since $Z_\ell$ is independent of
$\setacceptdeux_{\ell-1}(\ell)$, we have
\begin{align*}
\probaMarkovTilde{x,y}[\setacceptdeux_\ell(\ell)] \geq \delta
\probaMarkovTilde{x,y}[\setacceptdeux_{\ell-1}(\ell)] \ \probaMarkovTilde{x,y}[\sqrt{1-\rho^2}\norm{Z_{\ell}} \leq L/\ell] \eqsp.
\end{align*}
An immediate induction leads to $\probaMarkovTilde{x,y}[\setacceptdeux_\ell(\ell)]  \geq
\left(\probaMarkovTilde{x,y}[\sqrt{1-\rho^2}\norm{Z_1} \leq  L/\ell] \right)^\ell \delta^\ell$.
Plugging this result in \eqref{eq:final_pCN} and \eqref{eq:pCN_final_2} implies there exists $\zeta \in \ooint{0,1}$ such that  for all $x,y \in \boulefermee{0}{L}$, $Q^\ell d_\eta(x,y) \leq \zeta d_\eta(x,y)$.
\end{proof}


%% file: appendice.tex
\section{Wasserstein distance: some useful properties}

Let $(E, d)$ be a Polish space, with $d$ bounded by $1$. Then, for all $\mu,\nu \in \Pens(E)$:
$W_{d}(\mu,\nu) \leq   W_{d_0}(\mu,\nu)$ since for all $x,y \in E$, $d(x,y) \leq d_0(x,y)$.
Hence when $d$ is bounded by $1$, the convergence in total variation distance implies
the convergence in the Wasserstein metric $W_d$.



\begin{lemma}
  \label{lem:contract_coupling_weak_contraction}
  Let $(E,d)$ be a Polish space, with $d$ bounded by $1$, and let $P$ be a
  Markov kernel on $(E, \B(E))$. Let $Q$ be a coupling kernel for $P$.
\begin{enumerate} [(i)]
\item \label{lem:contraction_rev_mu_nu_2} Then, for all probability measures
  $\mu, \nu \in \Pens(E)$ and $n \in \nset^*$,
\[
W_d(\mu P^n, \nu P^n) \leq \inf_{\lambda \in \couplage{\mu}{\nu}} \int_{E \times E} Q^nd(z,t) \rmd \lambda(z,t) \eqsp.
\]
\item \label{lem:contract_rev_3} If in addition $Q$ is a $d$-weak-contraction,
  then for all $x,y \in E$, $ W_d(P(x,\cdot),P(y,\cdot)) \leq d(x,y)$ and for
  all probability measures $\mu, \nu \in \Pens(E)$,
\begin{equation*}
W_d(\mu P , \nu P ) \leq W_d (\mu , \nu) \eqsp.
\end{equation*}
\end{enumerate}
\end{lemma}
\begin{proof}
  \eqref{lem:contraction_rev_mu_nu_2} For every $\lambda \in
  \couplage{\mu}{\nu}$, $\lambda Q^n$ is a coupling of $\mu P^n$ and $\nu P^n$.
  This yields the result.  Consider now \eqref{lem:contract_rev_3}. Using
  \eqref{lem:contraction_rev_mu_nu_2}, we get
\begin{equation*}
W_d(\mu P, \nu P)
\leq \inf_{\lambda \in \couplage{\mu}{\nu}} \int_{E \times E} Q d(z,t) \rmd \lambda(z,t) \leq \inf_{\lambda \in \couplage{\mu}{\nu}}   \int_{E \times E} d(z,t) \rmd \lambda(z,t) \leq W_d(\mu,\nu) \eqsp.
\end{equation*}
\end{proof}

\section{Subgeometric functions and sequences}
\begin{lemma}
\label{lem:suite_sous_geometrique}
Let $r \in \Lambda_0$ and $R$ be given by \eqref{eq:definition:R}.
\begin{enumerate}[(i)]
\item \label{lem:suite_sous_geo_1} For all $t,v \in \R_+$, $r(t+v) \leq r(t) r(v)$.
\item \label{lem:suite_sous_geo_2} $R$ is differentiable, convex and increasing
  to $\plusinfty$.
\item \label{lem:suite_sous_geo3} $ \lim_{t \to \infty} r(t) / R(t) =0$.
\item \label{lem:suite_sous_geo_4} There exists a constant $C$ such that for
  any $t,v \in \rset_+$, $R(t+v) \leq C R(t) R(v)$. 
\item \label{lem:suite_sous_geo_5}  $\sup_k R(k) / \sum_{i=0}^{k-1}r(i)
  < \infty$.
\end{enumerate}
\end{lemma}

\begin{proof}
  \eqref{lem:suite_sous_geo_1} follows from \cite[Lemma~1]{Stone_sous_geo}. Consider now \eqref{lem:suite_sous_geo_2}.
  By definition, $r$ is
  non-decreasing, thus is bounded on every compact set; then, $R$ is
  continuous.  Moreover, it is differentiable and its derivative is $r$, which
  is non-decreasing. Then $R$ is convex. In addition $r(0) \geq 2$, thus $R$ is
  increasing to $\plusinfty$.
\eqref{lem:suite_sous_geo3}. Set $u(t) \eqdef \log(r(t))/t$. Since $r \in \Lambda_0$, the function $u$ is non increasing, which implies that, for every $h \in \ooint{0,1}$,
\begin{equation*}
\log \parenthese{1+ \{r(t+h)-r(t)\}/r(t)} =
\log \parenthese{r(t+h)/r(t)} = t(u(t+h) - u(t)) +  hu(t+h) \leq  hu(t+h) \eqsp.
\end{equation*}
 Since
$\lim_{t \to \plusinfty} u(t) = 0$, for all $\epsilon >0$, there exists $T\in
\rset_+$ such that for all $t \geq T$ and $h \in \ooint{0,1}$,
$\parenthese{r(t+h)-r(t)} \leq \epsilon h r(t) $.
Therefore for all $t \geq T$ and $h \in \ooint{0,1}$, $(R(t+h) - R(t) )/(h
R(t)) \leq \epsilon + r(T+1) / R(t)$.  Taking $h \rightarrow 0$ it follows $
r(t) / R(t) \leq \epsilon + r(T+1)/R(t)$, for all $t \geq T$. The proof is
concluded by \eqref{lem:suite_sous_geo_2}.  \eqref{lem:suite_sous_geo_4}
follows from \eqref{lem:suite_sous_geo_1} and \eqref{lem:suite_sous_geo3}.
Finally, for \eqref{lem:suite_sous_geo_5}, the upper bound follows from
\eqref{lem:suite_sous_geo_4} and $R(k-1) \leq 1+ \sum_{i=0}^{k-1}r(i)$.
\end{proof}

%% file: main.bbl
\begin{thebibliography}{10}

\bibitem{andrieu:fort:vihola:2014}
C.~Andrieu, G.~Fort, and M.~Vihola.
\newblock {Quantitative convergence rates for sub-geometric Markov chains}.
\newblock {\em Adv. Appl. Probab.}, 2014.
\newblock accepted for publication.

\bibitem{baxendale:2005}
P.~H. Baxendale.
\newblock Renewal theory and computable convergence rates for geometrically
  ergodic {M}arkov chains.
\newblock {\em Ann. Appl. Probab.}, 15(1B):700--738, 2005.

\bibitem{beskos:roberts:stuart:Voss:2008}
A.~Beskos, G.~Roberts, A.~Stuart, and J.~Voss.
\newblock M{CMC} methods for diffusion bridges.
\newblock {\em Stoch. Dyn.}, 8(3):319--350, 2008.

\bibitem{billingsley:1999}
P.~Billingsley.
\newblock {\em Convergence of probability measures}.
\newblock Wiley Series in Probability and Statistics: Probability and
  Statistics. John Wiley \& Sons, Inc., New York, second edition, 1999.
\newblock A Wiley-Interscience Publication.

\bibitem{bogachev:1998}
V.I. Bogachev.
\newblock {\em Gaussian Measures}.
\newblock Mathematical surveys and monographs. American Mathematical Society,
  1998.

\bibitem{butkovsky:2012}
O.~Butkovsky.
\newblock Subgeometric rates of convergence of {M}arkov processes in the
  {W}asserstein metric.
\newblock {\em Ann. Appl. Probab.}, 24(2):526--552, 2014.

\bibitem{butkovsky:veretennikov:2013}
O.~A. Butkovsky and A.~Yu. Veretennikov.
\newblock On asymptotics for {V}aserstein coupling of {M}arkov chains.
\newblock {\em Stochastic Process. Appl.}, 123(9):3518--3541, 2013.

\bibitem{cloez:hairer}
B.~Cloez and M.~Hairer.
\newblock Exponential ergodicity for {M}arkov processes with random switching.
\newblock {\em Bernoulli}, 2014.
\newblock To appear.

\bibitem{douc:fort:moulines:soulier:2004}
R.~Douc, G.~Fort, \'E. Moulines, and P.~Soulier.
\newblock Practical drift conditions for subgeometric rates of convergence.
\newblock {\em Ann. Appl. Probab.}, 14(3):1353--1377, 2004.

\bibitem{douc:guillin:moulines:2008}
R.~Douc, A.~Guillin, and E.~Moulines.
\newblock Bounds on regeneration times and limit theorems for subgeometric
  {M}arkov chains.
\newblock {\em Ann. Inst. Henri Poincar\'e Probab. Stat.}, 44(2):239--257,
  2008.

\bibitem{Douc07computableconvergence}
R.~Douc, E.~Moulines, and P.~Soulier.
\newblock Computable convergence rates for sub-geometric ergodic {M}arkov
  chains.
\newblock {\em Bernoulli}, pages 831--848, 2007.

\bibitem{fort:moulines:2003}
G.~Fort and E.~Moulines.
\newblock Polynomial ergodicity of {M}arkov transition kernels.
\newblock {\em Stochastic Process. Appl.}, 103(1):57--99, 2003.

\bibitem{WeakHarris}
M.~Hairer, J.C. Mattingly, and M.~Scheutzow.
\newblock {Asymptotic coupling and a general form of Harris' theorem with
  applications to stochastic delay equations}.
\newblock {\em Probability Theory and Related Fields}, 149(1-2):223--259, 2011.

\bibitem{hairer:stuart:vollmer:2012}
M.~Hairer, A.M. Stuart, and S.J. Vollmer.
\newblock {Spectral gaps for Metropolis-Hastings algorithms in infinite
  dimensions}.
\newblock {\em Ann. Appl. Probab.}, 24:2455--290, 2014.

\bibitem{jarner:roberts:2002}
S.~F. Jarner and G.~O. Roberts.
\newblock Polynomial convergence rates of {M}arkov chains.
\newblock {\em Ann. Appl. Probab.}, 12(1):224--247, 2002.

\bibitem{jarner:tweedie:2003}
S.~F. Jarner and R~.L. Tweedie.
\newblock Necessary conditions for geometric and polynomial ergodicity of
  random-walk-type.
\newblock {\em Bernoulli}, 9(4):559--578, 08 2003.

\bibitem{Jarner_Wasserstein}
S.F. Jarner and R.L. Tweedie.
\newblock {Locally contracting iterated functions and stability of {M}arkov
  chains.}
\newblock {\em J. Appl. Probab.}, 38(2):494--507, 2001.

\bibitem{madras:sezr:2010}
N.~Madras and D.~Sezer.
\newblock Quantitative bounds for {M}arkov chain convergence: {W}asserstein and
  total variation distances.
\newblock {\em Bernoulli}, 16(3):882--908, 2010.

\bibitem{bible}
S.~Meyn and R.~Tweedie.
\newblock {\em {M}arkov Chains and Stochastic Stability}.
\newblock Cambridge University Press, New York, NY, USA, 2nd edition, 2009.

\bibitem{roberts:rosenthal:2001}
G.~O. Roberts and J.~S. Rosenthal.
\newblock Small and pseudo-small sets for {M}arkov chains.
\newblock {\em Stoch. Models}, 17(2):121--145, 2001.

\bibitem{roberts:rosenthal:2004}
G.~O. Roberts and J.~S. Rosenthal.
\newblock General state space {M}arkov chains and {MCMC} algorithms.
\newblock {\em Probab. Surv.}, 1:20--71, 2004.

\bibitem{roberts:tweedie-Geom:1996}
G.~O. Roberts and R.~L. Tweedie.
\newblock {Geometric convergence and central limit theorems for
  multidimensional Hastings and Metropolis algorithms}.
\newblock {\em Biometrika}, 83(1):95--110, 1996.

\bibitem{Stone_sous_geo}
C.~Stone and S.~Wainger.
\newblock One-sided error estimates in renewal theory.
\newblock {\em Journal d’Analyse Mathématique}, 20(1):325--352, 1967.

\bibitem{Tuominen_subgeo}
P.~Tuominen and R.~L. Tweedie.
\newblock Subgeometric rates of convergence of f -ergodic {M}arkov chains.
\newblock {\em Adv. in Appl. Probab.}, page 775–798, 1994.

\bibitem{vertennikov:1997}
A.~Yu. Veretennikov.
\newblock On polynomial mixing bounds for stochastic differential equations.
\newblock {\em Stochastic Process. Appl.}, 70(1):115--127, 1997.

\bibitem{VillaniTransport}
C.~Villani.
\newblock {\em Optimal transport : old and new}.
\newblock Grundlehren der mathematischen Wissenschaften. Springer, Berlin,
  2009.

\end{thebibliography}
